\documentclass{amsart}
\usepackage{amsmath,amssymb,epsfig,amscd,amsthm}

\usepackage{verbatim}

\numberwithin{equation}{section}

\newtheorem{lemma}[equation]{Lemma}
\newtheorem{thm}[equation]{Theorem}
\newtheorem{conjecture}[equation]{Conjecture}
\newtheorem{cor}[equation]{Corollary}
\newtheorem{prop}[equation]{Proposition}

\newtheorem{question}[equation]{Question}

\theoremstyle{remark}

\newtheorem{remark}[equation]{Remark}
\newtheorem{remarks}[equation]{Remarks}

\newtheorem*{acknowledgments}{Acknowledgments}

\renewcommand{\bar}[1]{#1\llap{$\overline{\phantom{\rm#1}}$}}
\newcommand{\Aberk}[1]{\ensuremath{\bA_{\mathrm{Berk},#1}^1}}

\newcommand{\lra}{\longrightarrow}

\DeclareMathOperator{\hhat}{{\widehat{h}}}
\DeclareMathOperator{\supp}{{Supp}}
\DeclareMathOperator{\gal}{{Gal}}
\newcommand{\bA}{{\mathbb A}}
\newcommand{\N}{{\mathbb N}}

\newcommand{\Q}{{\mathbb Q}}
\newcommand{\R}{{\mathbb R}}
\newcommand{\C}{{\mathbb C}}
\newcommand{\E}{{\mathbb E}}
\newcommand{\M}{{\mathbb M}}

\newcommand{\Kbar}{{\bar{K}}}
\newcommand{\Qbar}{\bar{\mathbb{Q}}}

\newcommand{\bC}{{\mathbb C}}

\newcommand{\Gal}{{\rm Gal}}

\newcommand{\cV}{\mathcal{V}}
\newcommand{\cS}{\mathcal{S}}

\newcommand{\into}{\hookrightarrow}

\newcommand{\cD}{\mathcal{D}}

\renewcommand{\l}{\lambda}

\newcommand{\g}{\gamma}
\renewcommand{\d}{\delta}

\newcommand{\f}{\varphi}

\newcommand{\bfa}{{\mathbf a}}
\newcommand{\bfb}{{\mathbf b}}
\newcommand{\bfc}{{\mathbf c}}

\newcommand{\bff}{{\mathbf f}}
\newcommand{\bfg}{{\mathbf g}}

\newcommand{\Prep}{{\rm Prep}}
\newcommand{\Prepd}{{\rm PrepDiff}}

\newcommand{\bfM}{{\mathbf M}}

\newcommand{\cL}{\mathcal{L}}
\newcommand{\cX}{\mathcal{X}}
\newcommand{\cY}{\mathcal{Y}}

\newcommand{\bP}{{\mathbb P}}

\begin{document}



\title{Preperiodic points for families of polynomials}

\author{D.~Ghioca}
\address{
Dragos Ghioca\\
Department of Mathematics\\
University of British Columbia\\
Vancouver, BC V6T 1Z2\\
Canada
}
\email{dghioca@math.ubc.ca}

\author{L.-C.~Hsia}
\address{
Liang-Chung Hsia\\
Department of Mathematics\\ 
National Central University\\
Chung-Li, Taiwan, ROC
}
\email{hsia@math.ncu.edu.tw}

\author{T.~J.~Tucker}
\address{
Thomas Tucker\\
Department of Mathematics\\
University of Rochester\\
Rochester, NY 14627\\
USA
}
\email{ttucker@math.rochester.edu}

\thanks{The first author was partially supported by an NSERC Discovery Grant. The second
  author was partially supported by the National Center of Theoretical Sciences of Taiwan and 
 NSC Grant  99-2115-M-008-007-MY3.  The third   author was partially  
supported by NSF Grants 0801072 and 0854839.}


 \begin{abstract}
 Let $a(\l),b(\l)\in\C[\l]$ and let $f_{\lambda}(x)\in \C[x]$ be a 
 one-parameter family of polynomials indexed by all $\lambda\in
 \C$. We study whether there exist
 infinitely many $\lambda\in \C$ such that both $a(\l)$ and $b(\l)$ are
 preperiodic for $f_{\lambda}$.  
 \end{abstract}

\date{\today}

\maketitle

\section{Introduction}
\label{intro}

The classical Manin-Mumford conjecture for abelian varieties (now a
theorem due to Raynaud \cite{Ray1, Ray2}) predicts that the set of
torsion points of an abelian variety $A$ defined over $\C$ is not
Zariski dense in a subvariety $V$ of $A$, unless $V$ is a translate of
an algebraic subgroup of $A$ by a torsion point.  Pink and Zilber
extended the Manin-Mumford conjecture to a more general question
regarding unlikely intersections between a subvariety $V$ of a
semiabelian variety $A$ and families of algebraic subgroups of $A$ of
codimension greater than the dimension of $V$ (see \cite{BMZ, Habegger, M-Z-1, Pink}). Here we state a special case of the
conjecture when $V$ is a curve.

\begin{conjecture}[Pink-Zilber]
\label{Pink-Zilber conjecture}
Let $\cS$ be a semiabelian scheme over a variety $\cY$ defined over $\C$,
and let $V\subset \cS$ be a curve which is not contained in any proper algebraic subgroup of $\cS$. We define
$$\cS^{[2]}:=\bigcup_{y\in\cY} B_y,$$
where $B_y$ is the union of all algebraic subgroups of the fibre $\cS_y$ of codimension at least equal to $2$. 
Then the intersection of $V$ with $\cS^{[2]}$ is finite.
\end{conjecture}

In \cite{M-Z-1, M-Z-2}, Masser and Zannier study
Conjecture~\ref{Pink-Zilber conjecture} when $\cS$ is the square of the Legendre family of
elliptic curves $E_{\lambda}$ (over the base $\bA^1\setminus\{0,1\}$) given by the equation
$y^2=x(x-1)(x-\lambda)$. They show that there exist at most finitely
many $\lambda\in\C$ such that both
$P_{\l}:=\left(2,\sqrt{2(2-\lambda)}\right)$ and
$Q_{\l}:=\left(3,\sqrt{6(3-\lambda)}\right)$ are torsion points for
$E_{\lambda}$. Their result is a special case of Conjecture~\ref{Pink-Zilber conjecture} since one can show that the curve $\{(P_{\lambda},Q_{\lambda})\text{ : }\l\ne 0,1\}$ is not contained in a proper algebraic subgroup of $\cS$.

The result of Masser and Zannier has a distinctive dynamical
flavor. Indeed, one may consider the following more general
problem. Let $\{X_{\l}\}$ be an algebraic family of quasiprojective
varieties defined over $\C$, let $\Phi_{\l}:X_{\l}\lra X_{\l}$ be an
algebraic family of endomorphisms, and let $P_{\l}\in X_{\l}$ and
$Q_{\l}\in X_{\l}$ be two algebraic families of points. Under what
conditions there exist infinitely many $\l$ such that both $P_{\l}$
and $Q_{\l}$ are preperiodic for $\Phi_{\l}$? Indeed, the problem from
\cite{M-Z-1, M-Z-2} fits into this general dynamical framework by
letting $X_{\l}=E_{\l}$ be the Legendre family of elliptic curves, and
letting $\Phi_{\l}$ be the multiplication-by-$2$-map on each elliptic
curve in this family.

In \cite{Matt-Laura}, Baker and DeMarco study an interesting special case of the above general dynamical question. Given the complex numbers $a$ and $b$, and given the integer $d\ge 2$,
when there exist infinitely many $\lambda\in\C$ such that both $a$ and
$b$ are preperiodic for the action of $f_{\lambda}(x):=x^d+\lambda$ on
$\C$?  
They show that  this happens if and only if $a^d=b^d$. 

The family of polynomials $\{x^d+\lambda\}_{\lambda\in\C}$ from
\cite{Matt-Laura} is a family of polynomials in \emph{normal form} of
degree $d\ge 2$. We say that a polynomial $f(x)$ of degree $d$ is in normal
form, if it is monic and its coefficient of $x^{d-1}$ equals $0$. We
note that each polynomial of degree $d$ can be conjugated by a linear
polynomial $\delta$ such that $\delta^{-1}\circ f\circ \delta$ is a
polynomial in normal form. Indeed, if the leading coefficient of
$f(x)$ equals $c_d$, while its coefficient of $x^{d-1}$ equals
$c_{d-1}$, we may let $\delta(x)=c_d^{-1/(d-1)}x-c_{d-1}/d$; then
$\delta^{-1}\circ f\circ \delta$ is in normal form.  It is natural to
ask the question from \cite{Matt-Laura} for any family of polynomials
in normal form acting on the affine line whose coefficients are
parameterized by any set of one-variable polynomials.

One might hope to formulate a general dynamical version of
Conjecture~\ref{Pink-Zilber conjecture} for polarizable endomorphisms of projective varieties 
more general than multiplication-by-$m$ maps on abelian varieties (an 
endomorphism $\Phi$ of a projective variety $X$ is polarizable, if
there exists  $d\ge 2$ and a line bundle $\cL$ on $X$ such
that $\Phi^*(\cL)$ is linearly equivalent to $\cL^{\otimes d}$ in
${\rm Pic}(X)$) by using the analogy between abelian subschemes
and preperiodic subvarieties.  Unfortunately, there are already
counterexamples, even in the case of a constant base, as is
illustrated in \cite{IMRN}.  On the other hand, in the special case $X=\bP^1\times \bP^1$, these counterexamples
are well-understood, and we believe it is reasonable to ask if the following dynamical
Pink-Zilber conjecture may hold for families of maps $(f_{\l},f_{\l})$ acting on
$\bP^1\times\bP^1$.

\begin{question}
\label{main conjecture}
Let $Y$ be any quasiprojective curve defined over $\C$, and let $F$ be
the function field of $Y$.  Let $\bfa,\bfb\in\bP^1(F)$, and let
$V\subset \cX:=\bP^1_F\times_F\bP^1_F$ be the curve $(\bfa,\bfb)$.
Let $\bff:\bP^1\lra \bP^1$ be a rational map of degree $d \ge 2$
defined over $F$. Then for all but finitely many $\l\in Y$, $\bff$ induces
a well-defined rational map $f_{\l}:\bP^1\lra \bP^1$ defined over
$\C$.  If there exist infinitely many $\l\in Y$ such that both
$\bfa(\l)$ and $\bfb(\l)$ are preperiodic points of $\bP^1(\C)$ under
the action of $f_\l$, then must $V$ be contained in a proper
preperiodic subvariety of $\cX$ under the action of
$\Phi:=(\bff,\bff)$?
\end{question}

\begin{remark}
\label{discussions around Pink-Zilber}
One may ask a very similar question for the case the action on $\bP^1\times \bP^1$ is given by $\Phi:=(\bff,\bfg)$, for two families of rational maps which  contain \emph{no} Latt\`{e}s maps. Constant families of Latt\`{e}s maps constitute counterexamples to a more general extension of Question~\ref{main conjecture} to all families of rational maps $(\bff,\bfg)$ acting on $\bP^1\times\bP^1$. However, we believe that understanding the case of the same family of rational maps acting coordinatewise on $\bP^1\times\bP^1$ (as in Question~\ref{main conjecture}) would be crucial in any further extension of this dynamical Pink-Zilber problem.
\end{remark}

The transversal periodic subschemes of $\cX=\bP^1\times \bP^1$ under the action of $\Phi=(\bff,\bff)$ are
defined by equations of the following form in the set of variables
$(x,y)$ of $\cX$:
\begin{enumerate}
\item[(1)] $x=\bfc$ is a periodic point for $\bff$; or
\item[(2)] $y=\bfc$ is a periodic point for $\bff$; or
\item[(3)] $\f_1(x)=\f_2(y)$ for some maps $\f_i$ which commute with a power of $\bff$.
\end{enumerate}
(See \cite{Medvedev-Scanlon} for an explicit description of plane
curves fixed by the action of $(f,g)$ given by two polynomials $f$ and
$g$.)  Often, in (3), a function $\bff$ commutes only with powers of itself;
this likely accounts for the simple form of the conclusion in our
Theorem~\ref{main result}.  

A special case of Question~\ref{main conjecture} is when $Y=\bA^1$,
$\bff\in R[x]$ where $R=\C[\l]$, and $\bfa,\bfb\in R$.  
In Theorem~\ref{main result} we provide a positive answer to Question~\ref{main conjecture} for the
 family of polynomials in normal form  
\begin{equation}
\label{normal form}
f_{\lambda}(x) = x^d+\sum_{i=0}^{d-2} c_i(\lambda)x^i \;\quad
\text{where $c_i(\lambda)\in \C[\lambda]$ for $i = 0, \ldots,
  d-2,$} 
\end{equation}
together with some mild restriction
on the polynomials $\bfa$ and $\bfb$. Furthermore, note that we do not exclude the case that each $c_i$ is a constant polynomial, in which case $\{f_{\l}\}_{\l\in\C}$ is a constant family of polynomials.

First, as a matter of notation,  we rewrite
\begin{equation}
\label{normal form 22}
f_{\l}(x)=P(x)+\sum_{i=1}^r Q_i(x)\cdot \lambda^{m_i},
\end{equation}
for some polynomial $P\in\C[x]$ in normal form of degree $d$, and some nonnegative integer $r$ and integers $m_0:=0< m_1\dots < m_r$, and some polynomials $Q_i\in\C[x]$ of degrees $0\le e_i\le d-2$. We do not exclude the case $r=0$, in which case $\{f_{\l}\}_{\l}$ is a constant family of polynomials.

Let $\bfa(\l), \bfb(\l)\in \C[\l].$ 
If $\bfa$ is preperiodic for $\bff$, i.e. $\bff^k(\bfa)=\bff^{\ell}(\bfa)$ for some $k\ne \ell$, then for each $\bfb$ one can show  that there are infinitely many $\l\in \C$ such that
 $\bfb(\l)$ (and thus also $\bfa(\l)$) is preperiodic for $f_\l$
  (see also Proposition~\ref{infinitely many preperiodic}). Therefore, we may assume
that $\bfa$ and $\bfb$ are not preperiodic for $\bff$.

\begin{thm}
\label{main result}
Let $\bff:= f_{\lambda}$ be the family of one-parameter polynomials (indexed by all $\lambda\in\C$) given by
$$f_{\lambda}(x):=x^d+\sum_{i=0}^{d-2} c_i(\lambda)x^i=P(x)+\sum_{j=1}^r Q_j(x)\cdot \lambda^{m_j},$$
as above (see \eqref{normal form} and \eqref{normal form 22}).  Let $\bfa,\bfb\in\C[\l]$, and assume there exist nonnegative integers $k$ and $\ell$ such that the following conditions hold
\begin{enumerate}
\item[(i)]
 $f_{\l}^k(\bfa(\l))$ and $f_{\l}^{\ell}(\bfb(\l))$ have the same degree and the same leading coefficient as polynomials in $\l$; and 
\item[(ii)] 
if $m=\deg_{\l}(f_{\l}^k(\bfa(\l)))=\deg_{\l}(f_{\l}^{\ell}(\bfb(\l)))$, then $m\ge m_r$.
\end{enumerate}
Then there exist infinitely many $\lambda\in\C$ such that both $\bfa(\l)$
and $\bfb(\l)$ are preperiodic points for $f_{\lambda}$ if and only if
$f_{\lambda}^k(\bfa(\l))=f_{\lambda}^{\ell}(\bfb(\l))$.
\end{thm}

\begin{remarks}
  \label{first important remark}
  (a) In Theorem~\ref{main result}, the $1$-dimensional
$\C$-scheme $(\bfa,\bfb)\subset \cX:=\bP^1_{\C(\l)}\times
_{\C(\l)}\bP^1_{\C(\l)}$ is contained in the $2$-dimensional
$\C$-subscheme $\cY$ of $\cX$ given by the equation:
$$\bff^k(x)=\bff^{\ell}(y),$$
where $(x,y)$ are the coordinates of $\cX$. Such a $\cY$ is
fixed by the action of $(\bff,\bff)$ on $\cX$, as predicted by
Question~\ref{main conjecture}.
\\
(b) It follows from the Lefschetz Principle that the same statements
in Theorem~\ref{main result} hold if we replace $\C$ by any other
algebraically closed complete valued field of characteristic 0.
\\
(c) We note that  if $\bfc\in \C[\l]$ has the property that there exists $k\in\N$ such that $\deg_{\l}(f_{\l}^k(\bfc(\l)))=m$ has the property (ii) from Theorem~\ref{main result}, then $\bfc$ is not preperiodic for $\bff$ (see Lemma~\ref{lem:degree in l}).
\\
(d) If $\bff$ is not a constant family, then it follows
from Benedetto's theorem \cite{Rob} that $\bfc\in \C[\l]$ is not 
preperiodic for $\bff$ if and only if there exists $k\in\N$ such that
$\deg_{\l}(f_{\l}^k(\bfc(\l))) \ge m_r.$ On the other hand,
if $\bff$ is a constant family of polynomials defined over $\C$, i.e. $r=0$ and $m_0=0$ in Theorem~\ref{main result}, then implicitly $m>0$ (otherwise the conclusion holds trivially). 
\end{remarks}

In particular, the case where $\bff$ is a constant family of polynomials yields the following result.
\begin{cor}
\label{Bogomolov 0}
Let $f\in\C[x]$ be a polynomial of degree $d\ge 2$, and let $\bfa,\bfb\in\C[\l]$ be two polynomials of same degree and with the same leading coefficient. If there exist infinitely many $\l\in\C$ such that both $\bfa(\l)$ and $\bfb(\l)$ are preperiodic for $f$, then $\bfa=\bfb$.  
\end{cor}

\begin{proof}
The result is an immediate consequence of Theorem~\ref{main result} once we observe, as before, that we may replace $f$ with a conjugate $\delta^{-1}\circ f\circ\delta$ of itself which is a polynomial in normal form. (Note that in this case we also replace $\bfa$ and $\bfb$ by $\delta^{-1}(\bfa)$ and respectively $\delta^{-1}(\bfb)$ which are also polynomials in $\lambda$ of same degree and same leading coefficient.)
\end{proof}

An interesting special case of Corollary~\ref{Bogomolov 0} is the following result.
\begin{cor}
\label{first corollary 0}
Let $f\in\C[x]$ be a polynomial of degree $d\ge 2$, let $g\in\C[x]$ be
a nonconstant polynomial, and let $c\in\C^*$. Then there exist at most
finitely many $\l\in\C$ such that both $g(\l)$ and $g(\l+c)$ are
preperiodic for $f$.
\end{cor}

Corollary~\ref{Bogomolov 0} provides a positive answer to a special
case of Zhang's Dynamical Manin-Mumford Conjecture, which states that
for a polarizable endomorphism $\Phi: X \lra X$ on a projective variety, the only
subvarieties of $X$ containing a dense of set of preperiodic points
are those subvarieties which are themselves preperiodic under $f$ (see
\cite[Conjecture 2.5]{Zhang} or \cite[Conjecture 1.2.1, Conjecture
4.1.7]{Zhang-lec} for details). This conjecture turns out to be false
in general (see \cite{IMRN}), but it may be true in many cases.  For
example, let $X:=\bP^1\times \bP^1$, and $\Phi(x,y):=(f(x),f(y))$ for
a polynomial $f$ of degree $d\ge 2$, and $Y$ be the Zariski closure in
$X$ of the set $\{(\bfa(z),\bfb(z))\text{ : }z\in\C\}$, where $\bfa,
\bfb\in\C[x]$ are polynomials of same degree and with the same
leading coefficient; Corollary~\ref{Bogomolov 0} implies that if $Y$
contains infinitely many points preperiodic under $\Phi$, then $Y$ is
the diagonal subvariety of $X$, and thus is itself preperiodic
under $\Phi$.  Corollary~\ref{Bogomolov 0} also has consequences for a case of a revised
Dynamical Manin-Mumford Conjecture \cite[Conjecture 1.4]{IMRN} (see
Section~\ref{smc} for details).

Theorem~\ref{main result} generalizes the main result of
\cite{Matt-Laura} in two ways. On the one hand,  in the case
$\bfa$ and $\bfb$ are both constant we can prove a generalization of
the main result from \cite{Matt-Laura} as follows. 
\begin{thm}
\label{extension to C}
Let $a,b\in\C$, let $d\ge 2$, and let $c_0,\cdots, c_{d-2}\in\C[\l]$ such that $\deg(c_0)>\deg(c_i)$ for each $i=1,\dots,d-2$. If there are infinitely many $\l\in\C$ such that both $a$ and $b$ are preperiodic for 
$$f_{\l}(x):=x^d+\sum_{i=0}^{d-2}c_i(\l)x^i,$$
then $f_{\l}(a)=f_{\l}(b)$.
\end{thm}

\begin{proof}
We apply Theorem~\ref{main result} for $\bfa(\l):=f_{\l}(a)$ and $\bfb(\l):=f_{\l}(b)$.
\end{proof}

Theorem~\ref{extension to C} yields the following generalization of the main result from \cite{Matt-Laura}.
\begin{cor}
\label{second corollary}
Let $f\in\C[x]$ be any polynomial of degree $d\ge 2$, and let $a,b\in\C$.  Then there exist infinitely many $\l\in\C$ such that both $a$ and $b$ are preperiodic for $f(x)+\l$ if and only if $f(a)=f(b)$. 
\end{cor}

\begin{proof}
Note that in this case we may drop the hypothesis that $f(x)$ is in normal form since, as explained in the proof of Corollary~\ref{Bogomolov 0}, we may conjugate $f(x)$ by a linear polynomial so that it becomes a polynomial in normal form. 
\end{proof}

On the other hand, using our Theorem~\ref{main result} we are able to treat the case
where the pair of points $\bfa$ and $\bfb$ depend algebraically on the
parameter. This answers a question raised by Silverman
mentioned in \cite[\S~1.1]{Matt-Laura}. 
Furthermore, by taking $\bff = f(x) + \lambda$, 
Theorem~\ref{main result} yields the following.

\begin{cor}
\label{first corollary}
Let $f\in\C[x]$ be any polynomial of degree $d\ge 2$, and let $\bfa, \bfb\in \C[\l]$ be polynomials such that $\bfa$ and $\bfb$ have the same degree and the same leading coefficient. Then there are infinitely many $\l\in\C$ such that both $\bfa(\l)$ and $\bfb(\l)$ are preperiodic under the action of $f(x)+\l$ if and only if $\bfa(\l)=\bfb(\l)$. 
\end{cor}

\begin{proof}
Firstly, the theorem is vacuously true if $\bfa$ and $\bfb$ are constant polynomials, since then they are automatically equal because they have the same leading coefficient. So, we may assume that $\deg(\bfa)=\deg(\bfb)\ge 1$.

Secondly, we conjugate $f(x)$  by some linear polynomial $\delta\in\C[x]$ such that $g:=\delta^{-1}\circ f\circ \delta$ is a polynomial in normal form. Then we apply Theorem~\ref{main result} to the family of  polynomials $g(x)+\delta^{-1}(\l)$ and to the starting points $\delta^{-1}(\bfa(\l))$ and $\delta^{-1}(\bfb(\l))$. Since $\bfa$ and $\bfb$ are polynomials of same positive degree and same leading coefficient, it is immediate to check that  conditions (i)-(ii) of Theorem~\ref{main result} hold for $k=\ell=0$. Therefore, $\bfa(\l)=\bfb(\l)$ as desired.
\end{proof}

An important special case of Corollary~\ref{first corollary} is the following result.
\begin{cor}
\label{first 1 corollary}
Let $f\in\C[x]$ be any polynomial of degree $d\ge 2$, let $g\in\C[x]$ be any nonconstant polynomial, and let $c\in\C^*$. Then there exist at most finitely many $\l\in\C$ such that both $g(\l)$ and $g(\l+c)$ are preperiodic for $f(x)+\l$. 
\end{cor}


We prove Theorem~\ref{main result} first for the case when both $\bfa$ and $\bfb$, and also each of the $c_i$'s have algebraic coefficients, and then we extend our proof to the general case. For the extension to $\C$, we use a result of Benedetto \cite{Rob} (see also its extension \cite{Matt-nonisotrivial} to arbitrary rational maps) which states that for a polynomial $f$ of degree at least equal to $2$ defined over a function field $K$ of finite transcendence degree over a subfield $K_0$, if $f$ is not isotrivial (i.e., $f$ is not conjugate to a polynomial defined over $\bar{K_0}$), then each $x\in\bar{K}$ is preperiodic if and only if its canonical height $\hhat_f(x)$ equals $0$. Strictly speaking, Benedetto's result is stated for function fields of transcendence degree $1$, but a simple inductive argument on the transcendence degree yields the result for function fields of arbitrary finite transcendence degree (see also \cite[Corollary 1.8]{Matt-nonisotrivial} where Baker extends Benedetto's result to rational maps defined over function fields of arbitrary finite transcendence degree).

Our results and proofs are inspired by the results of
\cite{Matt-Laura} so that the strategy for the proof of Theorem~\ref{main
  result} essentially follows the ideas in the paper \cite{Matt-Laura}. 
However there are significantly more technical difficulties in our proofs. 
The plan of our proof is to use the $v$-adic generalized
Mandelbrot sets introduced in \cite{Matt-Laura} for the family of
polynomials $f_{\l}$, and then  use the equidistribution result of
Baker-Rumely from \cite{Baker-Rumely}. A key ingredient is
Proposition~\ref{prop:local height and green's function} which says
that the canonical local height of the point in question at the place
$v$ is a constant multiple of the Green's function associated to the
$v$-adic generalized Mandelbrot set. Then the condition that
$\bfa(\l)$ ($\bfb(\l)$) are preperiodic is translated
to the condition that the heights $h_{\M_{\bfa}}(\l)$ ($h_{\M_{\bfb}}(\l)$
respectively) are zero, for the corresponding parameter $\l$. Therefore
the equidistribution result of Baker-Rumely can thus be applied to
conclude that the $v$-adic generalized Mandelbrot sets for $\bfa(\l)$
and $\bfb(\l)$  are the same for each place $v$ (see Section~\ref{future}
for a discussion on more general situations). 
Finally, we need to use an explicit formula for the Green's function
associated to the $v$-adic generalized Mandelbrot set corresponding to
an archimedean valuation $v$ to conclude that the desired equality of
$f_\l^k(\bfa(\l))$ and $f_\l^\ell(\bfb(\l))$ holds. The main technical
difficulties in our proof compared to the proof in \cite{Matt-Laura} are due to the fact that the starting points
$\bfa$ and $\bfb$ vary with the parameter $\l$; also, the analysis of the $v$-adic
generalized Mandelbrot sets and of their properties (such as
Lemma~\ref{M_a is bounded for archimedean v}) is more technical than
the corresponding discussion from \cite{Matt-Laura}.
Extra work is needed for the explicit description of the Green's 
function for a $v$-adic generalized Mandelbrot set (when $v$ is an
archimedean place)
due to the fact that in our case the polynomial $f_{\l}$ has arbitrary
(finitely) many critical points which vary with $\l$ in contrast to
the family of polynomials $x^d+\l$ from \cite{Matt-Laura}, which has  only one critical point
for the entire family.

Assuming each $c_i$ and also $\bfa$ and $\bfb$ have algebraic coefficients, the exact same proof we have yields stronger statements of Theorem~\ref{main result} and Corollaries~\ref{Bogomolov 0} and \ref{first corollary} allowing us to replace the hypothesis that there are infinitely many $\l\in\Qbar$ such that both $\bfa(\l)$ and $\bfb(\l)$ are preperiodic for $f_{\l}$ with the weaker condition that there exists an infinite sequence of $\l_n\in\Qbar$ such that 
$$\lim_{n\to\infty} \hhat_{f_{\l_n}}(\bfa(\l_n))+\hhat_{f_{\l_n}}(\bfb(\l_n))=0,$$
where for each $\l\in\Qbar$, $\hhat_{f_{\l}}$ is the canonical height constructed with respect to the polynomial $f_{\l}$ (for the precise definition of the canonical height with respect to a polynomial map, see Section~\ref{notation}). Therefore we can prove a special case of Zhang's Dynamical Bogomolov Conjecture (see \cite{Zhang-lec}).
\begin{cor}
\label{small points corollary}
Let $Y\subset \bP^1\times \bP^1$ be a curve which admits a parameterization given by $(\bfa(z),\bfb(z))$ for $z\in\C$, where $\bfa,\bfb\in\Qbar[x]$ are polynomials of same degree and with the same leading coefficient. Let $f\in\Qbar[x]$ be a polynomial of degree at least equal to $2$, and let $\Phi(x,y):=(f(x),f(y))$ be the diagonal action of $f$ on $\bP^1\times \bP^1$. If there exist an infinite sequence of points $(x_n,y_n)\in Y(\Qbar)$ such that
$$\lim_{n\to\infty}\hhat_f(x_n) +\hhat_f(y_n)=0,$$
then $\bfa=\bfb$. In particular, $Y$ is the diagonal subvariety of
$\bP^1\times \bP^1$ and thus is preperiodic under the action of
$\Phi$.
\end{cor}
\begin{remark}
In fact, this result holds not only over $\Qbar$ but over the algebraic
closure of any global function field $L$ (whose subfield of constants is $K$), as long as $f$ is not conjugate to a polynomial with coefficients in $\Kbar$.  
\end{remark}

Note that the second author, together with Baker, proved a similar
result \cite[Theorem 8.10]{Baker-Hsia} in the case $Y$ is a line;
i.e., if a line in $\bP^1\times \bP^1$ contains an infinite set of
points of small canonical height with respect to the coordinatewise
action of the polynomial $f$ on $\bP^1\times\bP^1$, then the line $Y$
is preperiodic under the action of $(f,f)$ on $\bP^1\times\bP^1$.

The plan of our paper is as follows. In Section~\ref{notation} we set
up our notation, while in Section~\ref{spaces} we give a brief
overview of Berkovich spaces. Then, in Section~\ref{preliminaries} we
introduce some basic preliminaries regarding the iterates of a generic
point $\bfc$ under $\bff$.  Section~\ref{Berkovich} contains
computations of the capacities of the generalized $v$-adic Mandelbrot
sets associated to a generic point $\bfc$ under the action of
$\bff$. In Section~\ref{complex analysis} we prove an explicit formula
for the Green's function for the generalized $v$-adic Mandelbrot sets
when $v$ is an archimedean valuation. We proceed with our proof of the
direct implication in Theorem~\ref{main result} in Section~\ref{proof
  of our main result} (for the case $f_{\l}\in\Qbar[x]$ and
$\bfa,\bfb\in\Qbar[x]$) and in Section~\ref{transcendental} (for the
general case). In Section~\ref{converse} we prove the converse
implication from Theorem~\ref{main result}.  Then, in
Section~\ref{smc} we prove Corollary~\ref{small points corollary} and
discuss the connections between our Question~\ref{main conjecture} and
the Dynamical Manin-Mumford Conjecture formulated by Ghioca, Tucker,
and Zhang in \cite{IMRN}.  Finally, we conclude our paper with a
discussion of possible extensions of Question~\ref{main conjecture} to
arbitrary dimensions in Section~\ref{future}.

\begin{acknowledgments} 
  The second author acknowledges the support from 2010-2011
  France-Taiwan Orchid Program which enabled him to attend the Summer
  School on Berkovich spaces held at Institut de Math\'{e}matiques de
  Jussieu where this project was initiated in the summer of 2010. The
  authors thank Matthew Baker and Joseph Silverman for several useful
  conversations.
\end{acknowledgments}

\section{Notation and preliminary}
\label{notation}

For any quasiprojective variety $X$ endowed with an endomorphism
$\Phi$, we call a point $x\in X$ preperiodic if there exist two
distinct nonnegative integers $m$ and $n$ such that
$\Phi^m(x)=\Phi^n(x)$, where by $\Phi^i$ we always denote the
$i$-iterate of the endomorphism $\Phi$. If $n=0$ then, by convention, $\Phi^0$ is the identity map.

Let $K$ be a field of characteristic $0$ equipped with a set of inequivalent
absolute values (places) $\Omega_K$, normalized so that the product
formula holds;  more precisely, for each $v\in\Omega_K$ there exists a
positive integer $N_v$ such that for all $\alpha\in K^{\ast}$ we have
$\prod_{v\in \Omega}\,|\alpha|_v^{N_v} = 1$ where for $v\in \Omega_K$, the
corresponding absolute value is denoted by $|\cdot |_v$.  
Let $\C_v$ be a fixed completion of the algebraic closure of a completion of $(K,|\cdot |_v)$. When $v$ is an archimedean valuation, then $\C_v=\C$. 
We fix an extension of $|\cdot |_v$ to an absolute value of
$(\C_v,|\cdot |_v)$. Examples of
product formula fields (or \emph{global fields}) are number fields and function fields of
projective varieties which are regular in codimension 1 (see
\cite[\S~2.3]{lang} or \cite[\S~1.4.6]{bg06}). 

Let $f\in \C_v[x]$ be any polynomial of degree
$d\ge 2$.  Following Call and Silverman \cite{Call-Silverman},
 for each  $x\in \C_v$, we define the
 \emph{local canonical height} of $x$ as follows   
$$\hhat_{f,v}(x):=\lim_{n\to\infty}
\frac{\log^+|f^n(x)|_v}{d^n},$$
where by $\log^+z$ we always denote $\log\max\{z,1\}$ (for any real number $z$). 

It is immediate that $\hhat_{f,v}(f^i(x))=d^i \hhat_{f,v}(x)$, and
thus $\hhat_{f,v}(x)=0$ whenever $x$ is a preperiodic point for $f$.
If $v$ is nonarchimedean and $f(x)=\sum_{i=0}^d a_i x^i$, then
$|f(x)|_v=|a_d x^d|_v>|x|_v$ when $|x|_v>r_v$, where
\begin{equation}
\label{definition N_v}
r_v:=\max\left\{|a_d|_v^{-\frac{1}{d-1}},\max\left\{
    \left|\frac{a_i}{a_d}\right|^{\frac{1}{d-i}}\right\}_{0\le
    i<d}\right\}.
\end{equation}
Moreover, if $|x|_v>r_v$, then
$\hhat_v(x)=\log|x|_v+\frac{\log|a_d|_v}{d-1}>0$. 
For more details see \cite{equi} 
and \cite{L-C} (although the results from \cite{equi, L-C} are for
canonical heights associated to Drinfeld modules, all the proofs go
through for any local canonical height associated to any polynomial
with respect to any nonarchimedean place). 

Now, if $v$ is archimedean, again it is easy to see that if $|x|_v$ is
sufficiently large
then $|f(x)|_v\gg |x|^d_v$ and
moreover, $|f^n(x)|_v\to\infty$ as $n\to\infty$. 

We fix an algebraic closure $\Kbar$ of $K$, and for each $v\in\Omega_K$ we fix an embedding $\Kbar\into\C_v$. Assume $f\in\Kbar[x]$. 
In \cite{Call-Silverman}, Call and Silverman also defined the \emph{global canonical height} $\hhat(x)$ for each $x\in\Kbar$ as
$$\hhat_f(x)=\lim_{n\to\infty} \frac{h(f^n(x))}{d^n},$$
where $h$ is the usual (logarithmic) Weil height on $\Kbar$.  Call and Silverman show that the global canonical height decomposes into a sum of the corresponding local canonical heights.

For each $\sigma\in\Gal(\Kbar/K)$, we denote by $\hhat_{f^{\sigma}}$ the global canonical height computed with respect to $f^{\sigma}$, which is the polynomial obtained by applying $\sigma$ to each coefficient of $f$. Similarly, for each $v\in\Omega_K$ we denote by $\hhat_{f^{\sigma},v}$ the corresponding local canonical height constructed with respect to the polynomial $f^{\sigma}$. For $x\in \overline{K}$, we have $\hhat_f(x)=0$ if and only if $\hhat_{f^{\sigma}}\left(x^{\sigma}\right)=0$ for all $\sigma\in\Gal(\Kbar/K)$. More precisely,  for $x\in \Kbar$ we have 
\begin{equation}
\label{connection between local and global canonical heights}
\hhat_f(x)=0\text{ if and only if }\hhat_{f^{\sigma},v}\left(x^{\sigma}\right)=0\text{ for all $v\in \Omega_K$ and all $\sigma\in{\rm Gal}(\Kbar/K)$.}
\end{equation}
Essentially, \eqref{connection between local and global canonical heights} says that $\hhat_f(x)=0$ if and only if the orbits of $x^{\sigma}$ under each polynomial $f^{\sigma}$ (for $\sigma\in\Gal(\Kbar/K)$) are bounded with respect to each absolute value $|\cdot |_v$ for $v\in\Omega_K$.

In \cite{Rob}, Benedetto proved that if a polynomial $f$ defined over
a function field $K$  (endowed with a set $\Omega_K$ of absolute
values) is not \emph{isotrivial} (that is, it cannot be conjugated to a
polynomial defined over the constant subfield of $K$) then each point
$c\in \Kbar$ is preperiodic for $f$ if and only if its global
canonical height (computed with respect to $f$) equals $0$. In particular, if 
$c\in \Kbar$, then $c$ is preperiodic if and only if 
\begin{equation}
\label{isotrivial thing}
\hhat_{f^{\sigma},v}\left(c^{\sigma}\right)=0\text{ for all $\sigma\in\Gal(\Kbar/K)$ and for all places $v\in\Omega_K$.} 
\end{equation}

Let $\bff=f_{\l}:=x^d+\sum_{i=0}^{d-2} c_i(\l)x^i$, where $c_i(\l)\in\C[\l]$  for $i=0,\dots,d-2$, and let $\bfc(\l)\in\C[\l]$. We let $K$ be the field extension of $\Q$ generated by all coefficients of each $c_i(\l)$ and of $\bfc(\l)$. Assume $K$ is a global field, i.e. it has a set $\Omega_K$ of inequivalent absolute values with respect to which the nonzero elements of $K$ satisfy a product formula. 
For each place $v\in\Omega_K$ we define the $v$-adic Mandelbrot set $M_{\bfc,v}$ for $\bfc$ with respect to the family of polynomials $\bff$ as the set of all $\l\in\C_v$ such that $\hhat_{f_{\l}, v}(\bfc(\l))=0$, i.e. the set of all $\l\in\C_v$ such that the iterates $f_{\l}^n(\bfc(\l))$ are bounded with respect to the $v$-adic absolute value.

\section{Berkovich spaces}
\label{spaces}


In this section we introduce the Berkovich spaces, and state the equidistribution theorem of Baker and Rumely \cite{Baker-Rumely} which will be key for the proofs of Theorems~\ref{main result}~and~\ref{extension to C}.

Let $K$ be a global field of characteristic $0$, and let $\Omega_K$ be the set of its inequivalent absolute values. For each $v\in\Omega_K$, we let $\C_v$ be the completion of an algebraic closure of the completion of $K$ at $v$.  
Let $\Aberk{\C_v}$ denote the Berkovich affine  line over
$\C_v$ (see \cite{berkovich} or  \cite[\S~2.1]{Baker-Rumely} for
details). Then $\Aberk{\C_v}$ is a 
locally compact, Hausdorff, path-connected space containing $\C_v$ as a
dense subspace (with the topology induced from the $v$-adic absolute
value). As a topological space, $\Aberk{\C_v}$ is the set consisting of
all multiplicative seminorms, denoted by $[\cdot]_x,$ on $\C_v[T]$ extending
the absolute value $|\cdot|_v$ on $\C_v$ endowed with the weakest
topology such that the map $z\mapsto [f]_z$ is continuous for all
$f\in \C_v[T]$. It follows from the Gelfand-Mazur theorem that if $\C_v$ is the field of complex numbers
 $\C$ then $\Aberk{\C}$ is homeomorphic to $\C$. In the following, we
will also use $\Aberk{\C_v}$ to denote the complex line $\C$ whenever $\C_v = \C$. If $(\C_v,|\cdot |_v)$ is nonarchimedean then 
the set of seminorms can be described as follows. If
$\{D(a_i,r_i)\}_i$ is any decreasing nested sequence of closed disks
$D(c_i,r_i)$ centered at points $c_i\in \C_v$ of radius $r_i\ge 0$, then
the map $f\mapsto \lim_{i\to\infty}\,[f]_{D(c_i,r_i)}$ defines a
multiplicative seminorm on $\C_v[T]$ where $[f]_{D(c_i,r_i)}$ is the
sup-norm of $f$ over the closed disk 
$D(a_i, r_i).$ Berkovich's classification theorem says that  
there are exactly four types of points, Type I, II, III and IV. 
The first three types of points can be described in terms of closed
disks $\zeta = D(c,r) = \cap D(c_i,r_i)$ where $c\in \C_v$ and $r\ge 0.$
The corresponding 
multiplicative seminorm is just $f \mapsto [f]_{D(c,r)}$ for $f\in
\C_v[T].$ Then, $\zeta$ is of Type I, II or III  if and only if $r= 0,
r\in |\C_v^{\ast}|_v$ or $r\not\in |\C_v^{\ast}|_v$ respectively. As for
Type IV points, they correspond to sequences of decreasing nested
disks $D(c_i,r_i)$ such that $\cap D(c_i,r_i) = \emptyset$ and the
multiplicative seminorm is $f\mapsto \lim_{i\to \infty}
[f]_{D(c_i,r_i)}$ as described above.  For details, see
\cite{berkovich} or \cite{Baker-Rumely}. 
For $\zeta\in \Aberk{\C_v},$ we sometimes write $|\zeta|_v$
instead of $[T]_{\zeta}.$ 

In order to apply the main equidistribution result from \cite[Theorem 7.52]{Baker-Rumely}, we recall the potential theory on the
affine line over $\C_v$. We will focus on the case $\C_v$ is a nonarchimedean
field; the case $\C_v=\C$ is classical (we refer the reader to \cite{Ransford}).  The right setting for nonarchimedean potential theory is the
potential theory on $\Aberk{\C_v}$ developed in \cite{Baker-Rumely}. We
quote part of a nice summary of the theory 
from \cite[\S~2.2~and~2.3]{Matt-Laura} without going into details. We
refer the reader to \cite{Matt-Laura, Baker-Rumely}  for all the
details and proofs. Let $E$ be a compact subset of
$\Aberk{\C_v}.$ Then analogous to the complex case, the logarithmic
capacity $\g(E) = e^{-V(E)}$ and the Green's function $G_E$ of $E$
relative to $\infty$ can be defined where $V(E)$ is the infimum of the 
\emph{energy integral} with respect to all possible probability
measures supported on $E$. More precisely,
$$V(E)=\inf_{\mu}\int\int_{E\times E} -\log\delta(x,y) d\mu(x)d\mu(y),$$
where the infimum is computed with respect to all probability measures $\mu$ supported on $E$, while $\delta(x,y)$ is the \emph{Hsia kernel} (see \cite{Baker-Rumely}):
$$\delta(x,y):=\limsup_{\substack{z,w\in\Aberk{\C_v}\\ z\to x, w\to y}}\mid z-w\mid _v .$$ 
The following are basic properties of the logarithmic 
capacity of $E$. 
\begin{itemize}
\item
  If $E_1, E_2$ are two compact subsets of $\Aberk{\C_v}$ such that
  $E_1\subset E_2$ then $\g(E_1) \le \g(E_2).$

\item
  If $E = \{\zeta\}$ where $\zeta$ is a Type II or III point
  corresponding to   a closed disk $D(c,r)$ then $\g(E) = r > 0.$ 
  \cite[Example~6.3]{Baker-Rumely}. (This can be viewed as analogue of
  the fact that a closed disk $D(c,r)$ of positive radius $r$ in $\C_v$
  has logarithmic capacity $\g(D(c,r)) = r$.)    
\end{itemize}
If $\g(E) > 0,$ then the exists a unique probability measure $\mu_E$
attaining the infimum of the energy integral. Furthermore, the support of
$\mu_E$ is contained in the boundary of the unbounded component of
$\Aberk{\C_v}\setminus E.$  

The Green's function 
$G_E(z)$ of $E$ relative to infinity is a well-defined nonnegative
real-valued subharmonic function on $\Aberk{\C_v}$ which is harmonic on $\Aberk{\C_v}\setminus E$ (in the sense of
\cite[Chapter~8]{Baker-Rumely} for the nonarchimedean setting; see \cite{Ransford} for the archimedean case). 
If $\g(E)=0$, then there exists no Green's function associated to the set $E$ (see \cite[Exercise 1, page 115]{Ransford} in the case $|\cdot |_v$ is archimedean; a similar argument works in the case $|\cdot |_v$ is nonarchimedean). Indeed, as shown in \cite[Proposition 7.17, page 151]{Baker-Rumely}, if $\gamma(\partial E)=0$ then  there exists no nonconstant harmonic function on $\Aberk{\C_v}\setminus E$ which is bounded below (this is the Strong Maximum Principle for harmonic functions defined on Berkovich spaces). 
The following result is \cite[Lemma~2.2~and~2.5]{Matt-Laura}, and it gives a characterization of
the Green's function of the set $E$.  

\begin{lemma}
  \label{green function}
Let $(\C_v,|\cdot |_v)$ be either an archimedean or a nonarchimedean field. 
Let $E$ be a compact subset of $\Aberk{\C_v}$ and let $U$ be the unbounded component of
$\Aberk{\C_v}\setminus E$. 

\begin{itemize}
\item[(1)]
If $\g(E)>0$ (i.e. $V(E)<\infty$), then
  $G_E(z) = V(E) + \log |z|_v + o(1)$ for all $z\in \Aberk{\C_v}$ such that
  $|z|_v$ is sufficiently large. Furthermore, the $o(1)$-term may be omitted if $v$ is
  nonarchimedean. 

\item[(2)]
If $G_E(z) = 0$ for all $z\in E,$ then $G_E$ is continuous on
  $\Aberk{\C_v},$ $\supp(\mu_E) = \partial U$ and $G_E(z) > 0$ if and
  only if $z\in U.$  

\item[(3)]
  If $G :\Aberk{\C_v}\to \R$ is a continuous subharmonic function which
  is harmonic on $U,$ identically zero on $E,$ and such that $G(z) -
  \log^+|z|_v$ is bounded, then $G = G_E$. Furthermore, if $G(z)=\log|z|_v + V + o(1)$ (as $|z|_v\to\infty$) for some $V<\infty$, then $V(E)=V$ and so, $\g(E)=e^{-V}$.

\end{itemize}

\end{lemma}

To state the equidistribution result from \cite{Baker-Rumely}, we
consider the compact \emph{Berkovich adelic sets} which are of the
following form
$$
\E := \prod_{v\in \Omega}\, E_v
$$
where $E_v$ is a non-empty compact subset of $\Aberk{\C_v}$ for each
$v\in \Omega$ and where $E_v$ is the closed unit disk $\cD(0,1)$ in
$\Aberk{\C_v}$ for all but finitely many $v\in \Omega.$ The
\emph{logarithmic capacity} $\g(\E)$ of $\E$ is defined as follows
$$\g(\E) = \prod_{v\in \Omega}\,\g(E_v)^{N_v}, $$
where the positive integers $N_v$ are the ones associated to the product formula on the global field $K$. 
Note that this is a finite product as for all but finitely many 
$v\in\Omega,$ $\g(E_v) = \g(\cD(0,1)) = 1.$   Let $G_v = G_{E_v}$  be
the Green's function of $E_v$ 
relative to $\infty$ for each $v\in \Omega.$ For every $v\in\Omega,$ we
fix an embedding $\Kbar\into\C_v.$ Let 
$S\subset \Kbar$ be any finite subset that is invariant under the action
of the Galois group 
$\gal(\Kbar/K)$. We define the height $h_{\E}(S)$ of $S$ relative to $\E$
by
\begin{equation}
\label{def ade hig}
h_{\E}(S) = \sum_{v\in\Omega} N_v\left(\frac{1}{|S|}\sum_{z\in S}G_v(z)\right).
\end{equation}
Note that this definition is independent of any particular embedding
$\Kbar\into\C_v$ that we choose at $v\in \Omega.$
The following is a special case of the equidistribution result 
\cite[Theorem~7.52]{Baker-Rumely} that we need for our application.
\begin{thm}
 \label{thm:equidistribution}
Let $\E = \prod_{v\in \Omega} E_v$ be a compact Berkovich adelic set
 with $\g(\E)=1.$ Suppose that 
 $S_n$ is a sequence of $\gal(\Kbar/K)$-invariant finite subsets of
 $\Kbar$ with $|S_n|\to \infty$ and $h_{\E}(S_n) \to 0$ as $n\to
 \infty.$ 
For each $v\in\Omega$ and for each $n$ let $\d_n$ be the discrete
 probability measure supported equally on the elements of $S_n.$ Then the
 sequence of measures $\{\d_n\}$ converges weakly to $\mu_v$ the
 equilibrium measure on $E_v.$
\end{thm}

\section{General results about the dynamics of  polynomials $f_{\lambda}$}  
\label{preliminaries}

In this Section we work with a family of polynomials $f_{\lambda}$
as given  in Section~\ref{notation}, i.e.
$$f_{\lambda}(x)=x^d + \sum_{i=0}^{d-2} c_i(\lambda) x^i,  $$
with $c_i(\lambda) \in \C[\lambda]$ for $i = 0, \ldots, d-2$. As before, we may rewrite our family of polynomials as
$$f_{\l}(x)=P(x)+\sum_{j=1}^r Q_j(x)\cdot \l^{m_j},$$
where $P(x)$ is a polynomial of degree $d$ in normal form, each $Q_i$ has degree at most equal to $d-2$, while $r$ is a nonnegative integer and $m_0:=0<m_1\cdots < m_r$. 
Let
$\bfc(\lambda) \in \C[\lambda]$ be given, and let $K$ be the field extension of $\Q$ generated by all
the coefficients of $c_i(\lambda), i = 0,\ldots,d-2$ and of $\bfc(\lambda).$  
We define $g_{\bfc,n}(\lambda):= f_{\lambda}^n(\bfc(\lambda))$ for each
$n\in\N$. Assume $m := \deg(\bfc)$ satisfies the property (ii) from Theorem~\ref{main result}, i.e.,
\begin{equation}
\label{what we need for degree}
m=\deg(\bfc)\ge m_r.
\end{equation} 
Furthermore, if $r=0$ we assume $m\ge 1$ (see also Remark~\ref{first important remark} (c)). 
We let $q_m$ be the leading coefficient of $\bfc(\l)$. 
In the next Lemma we compute the degrees of all polynomials $g_{\bfc,n}$ for all positive integers $n$.

\begin{lemma}
\label{lem:degree in l}
With the above hypothesis, the polynomial $g_{\bfc,n}(\lambda)$ has degree $m\cdot d^{n}$ and leading coefficient $q_m^{d^{n}}$ for each $n\in\N$.
\end{lemma}

\begin{proof}
The assertion  follows
easily by induction on $n$, using \eqref{what we need for degree}, since the term of highest degree in $\l$ from $g_{\bfc,n}(\l)$ is  $\bfc(\l)^{d^n}$.  
\end{proof}

We immediately obtain as a corollary of Lemma~\ref{lem:degree in l} the fact that $\bfc$ is not preperiodic for $\bff$. 
We denote by $\Prep(\bfc)$ the set of all $\l\in\C$ such that $\bfc(\l)$ is preperiodic for $f_{\l}$. The following result is an immediate consequence of Lemma~\ref{lem:degree in l}.
\begin{cor}
\label{always over the complex}
 $\Prep(\bfc)\subset \Kbar.$
\end{cor}

\section{Capacities of Generalized Mandelbrot sets}
\label{Berkovich}

We continue with the notation from Sections~\ref{spaces} and
\ref{preliminaries}.    
Let $\bfc=\bfc(\lambda)\in\C[\lambda]$ be a nonconstant polynomial, and let $K$ be a product formula
field containing the 
coefficients of each $c_i(\lambda), i = 0,\ldots, d-2$ and of $\bfc$.    
We let $\Omega_K$ be the set of inequivalent absolute values of the
global field $K$, and let $v\in \Omega_K$. Assume that $\bfc(\l) =
q_m \lambda^m + \ldots \text{(lower terms)},$ where $m=\deg(\bfc)$ satisfies the condition \eqref{what we need for degree}.

Our goal is to compute
 the logarithmic capacities of the $v$-adic generalized 
Mandelbrot sets $M_{\bfc,v}$ defined in Section~\ref{notation}. 
Following \cite{Matt-Laura},  we
extend the definition of our $v$-adic Mandelbrot set $M_{\bfc,v}$ to be a
subset of the affine Berkovich line $\Aberk{\C_v}$ as follows:
\[
 M_{\bfc,v} := \{\l \in \Aberk{\C_v} : \sup_n \left[g_{\bfc,n}(T)\right]_\l < \infty\}.
\]
Note that if $\C_v$ is a nonarchimedean field, then our present definition for $M_{\bfc,v}$ yields more points than our definition from 
Section~\ref{notation}. 
Let $\l\in \C_v$ and recall the
local canonical height $\hhat_{\l,v}(x)$ of $x\in \C_v$  is given by the formula
$$
\hhat_{\l,v}(x):=\hhat_{f_{\l},v}(x) =\lim_{n\to\infty}
\frac{\log^+|f_\l^n(x)|_v}{d^n}.
$$
Notice that $\hhat_{\l,v}(x)$ is a continuous function of both $\l$
and $x$ (see \cite[Prop.1.2]{BH88} for polynomials over
complex numbers; the proof for the nonarchimedean case is
similar). 
As $\C_v$ is a dense subspace of $\Aberk{\C_v}$,
continuity in $\l$ implies that  the canonical local height function 
$\hhat_{\l,v}(\bfc(\lambda))$ has a natural extension on $\Aberk{\C_v}.$ In the
following, we will view $\hhat_{\l,v}(\bfc(\lambda))$ as a continuous function on
$\Aberk{\C_v}.$ It follows from the definition of $M_{\bfc,v}$
that 
$\l\in M_{\bfc,v}$ if and only if $\hhat_{\l,v}(\bfc(\lambda)) = 0.$  Thus, 
$M_{\bfc,v}$ is a closed subset of $\Aberk{\C_v}.$ In fact, the following is
true. 

\begin{prop}
\label{prop:bounded mandelbrot} 
$M_{\bfc,v}$ is a compact subset of $\Aberk{\C_v}.$ 
\end{prop}

We already showed that $M_{\bfc,v}$ is a closed subset of the locally
compact space $\Aberk{\C_v}$, and thus in order to prove
Proposition~\ref{prop:bounded mandelbrot} we only need to show that
$M_{\bfc,v}$ is a bounded subset of $\Aberk{\C_v}$. If $\bff$ is a
constant family of polynomials, then Proposition~\ref{prop:bounded
  mandelbrot} follows from our assumption that $\deg(\bfc)\ge 1$.
Indeed, if $|\lambda|_v$ is large, then
$|\bfc(\l)|_v$ is large and thus $|\bff^n(\bfc(\l))|_v\to\infty$ as
$n\to\infty$.
Furthermore, for nonarchimedean place $v$ if
$|\l|_v$ is sufficiently large, then (assuming $v$ is nonarchimedean)
\begin{equation}
\label{0 equation 1}
|\bff^n(\bfc(\l))|_v=|\bfc(\l)|_v^{d^n}=|q_m\l^m|_v^{d^n}.
\end{equation} 

So, now we are left with the case that $\bff$ is not a constant family, i.e. $r\ge 1$.
\begin{lemma}
\label{M_a is bounded for archimedean v}
Assume $r\ge 1$, i.e. $\bff$ is not a constant family of polynomials. Then $M_{\bfc,v}$ is a bounded subset of  $\Aberk{\C_v}$.
\end{lemma}

\begin{proof}
First we rewrite as before
$$f_{\l}(x)=P(x)+\sum_{j=1}^r Q_j(x)\cdot \l^{m_j},$$
with $P(x)$ in normal form of degree $d$, and each polynomial $Q_j$ of degree $e_j\le d-2$; also, $0<m_1<\cdots < m_r$. We know $m=\deg(\bfc)\ge m_r$.

Since $q_m\lambda^{m}$ is the leading monomial in $\bfc$,
there exists a positive real number $C_1$ depending only on $v$,
coefficients of $c_i(\lambda), i=0, \ldots, d-2$ and on
$\bfc$ such that if $|\lambda|_v>C_1$, then
$|\bfc(\lambda)|_v>\frac{|q_m|_v}{2}\cdot 
|\lambda|_v^{m}$.

Let $\alpha:=\max_{i=1}^{r} \frac{m_i}{d-e_i}$; then $\alpha\le m_r/2$ since $e_i\le d-2$ for all $i$. 
There exist positive real numbers $C_2$ and $C_3$ (depending only on
$v$, and on the coefficients of $c_i(\lambda)$) such that if $|\lambda|_v>C_2$, and if $|x|_v> C_3|\lambda|_v^{\alpha}$, then 
$$|f_{\lambda}(x)|_v>\frac{|x|_v^d}{2}>|x|_v,$$
and thus $|f_{\lambda}^n(x)|_v\to\infty$ as $n\to\infty$.

On the other hand, since $m\ge m_r\ge 2\alpha > \alpha$, we conclude
that if 
$|\lambda|_v>\left(2C_3/|q_m|_v\right)^{1/(m-\alpha)}$
then $$\frac{|q_m|_v}{2}\cdot
|\lambda|_v^{m}>C_3|\lambda|_v^{\alpha}.$$ 
We let $C_4:=\max\left\{C_1,C_2,\left(2C_3/|q_m|_v\right)^{1/(m-\alpha)}, |q_m|_v^{-1/m}\right\}$. So, if $|\lambda|_v>C_4$ then 
$$|\bfc(\l)|_v>\frac{|q_m|_v}{2}\cdot |\lambda|_v^{m}>C_3|\lambda|_v^{\alpha},$$ 
and thus $|f_{\lambda}^n(\bfc(\l))|_v\to\infty$ as $n\to\infty$.
We conclude that if $\lambda\in M_{\bfc,v}$, then $|\lambda|_v\le C_4$, as desired. 
\end{proof}

\begin{remark}
It is possible to make the constants in the above proof explicit. Moreover,
for a nonarchimedean place $v$ the estimate of the absolute values can be
precise. For example, if $v$ is nonarchimedean, we can ensure that if
$|\lambda|_v>C_4$, then  
\begin{equation}
\label{precise absolute value}
|f_{\lambda}^{n}(\bfc(\l))|_v=|q_m\lambda^{m}|_v^{d^n}\text{ for all $n\ge 1$.}
\end{equation}
\end{remark}

\begin{thm}
  \label{mandelbrot capacity}
The logarithmic capacity of
$M_{\bfc,v}$ is $\g(M_{\bfc,v}) = |q_m|_v^{-1/m} $. 
\end{thm}

The strategy for the proof of Theorem~\ref{mandelbrot capacity} is to
construct a continuous subharmonic function $G_{c,v} : \Aberk{\C_v}\to \R$
satisfying Lemma~\ref{green function}~(3).  
Analogous to the family $f_{\lambda}(x) = x^d + \lambda$ treated in
\cite{Matt-Laura} we let 

\begin{equation}
  \label{mandelbrot green function}
G_{\bfc,v}(\l) := \lim_{n\to \infty}\, \frac{1}{\deg(g_{\bfc,n})} \log^+ [g_{\bfc,n}(T)]_\l. 
\end{equation}
Then by a similar reasoning as in the proof of
\cite[Prop.~3.7]{Matt-Laura}, it can be shown that the limit exists
for all $\l\in \Aberk{\C_v}$.
In fact, by the definition of canonical local height, for $\l\in \C_v$
we have
\begin{align*}
  G_{\bfc,v}(\l) & = \lim_{n\to \infty}\, \frac{1}{md^{n}}\log^+ |f_\l^{n}(\bfc(\l))|_v \quad \text{since 
  $\deg(g_{\bfc,n}) = m d^{n}$ by Lemma~\ref{lem:degree in l}, } \\
&  = \frac{1}{m}\cdot \hhat_{f_{\l},v}(\bfc(\l)) \quad \text{by the definition of canonical local height.}
\end{align*}

As a consequence of the above computation, we have the following. 
\begin{prop}[c.f. {\cite[Theorem~II.0.1]{silverman94-1} and
    \cite[Theorem~III.0.1 and Corollary~III.0.3]{silverman94-2}}] 
  \label{prop:local height and green's function}
$$ \hhat_{f_{\l},v}(\bfc(\l)) = \deg(\bfc) G_{\bfc,v}(\l) .$$
\end{prop}

\begin{remark}
The above formula holds in the more general case of Question~\ref{main conjecture}; for example, one may work with a rational function $\bfc\in\C(\l)$.
\end{remark}

Note that $G_{\bfc,v}(\l) \ge 0$ for all $\l\in \Aberk{\C_v}.$ Moreover, 
we see easily that $\l\in M_{\bfc,v}$ if and only if $G_{\bfc,v}(\l) = 0.$ 

\begin{lemma}
\label{green's function for mandelbrot}
$G_{\bfc,v}$ is the Green's function for $M_{\bfc,v}$ relative to $\infty.$ 
\end{lemma}

The proof is essentially the same as the proof of
\cite[Prop.~3.7]{Matt-Laura}, we simply give a sketch of the idea. 

\begin{proof}[Proof of Lemma~\ref{green's function for mandelbrot}.]
We deal with the case that $v$ is nonarchimedean (the case when $v$ is archimedean follows similarly). So, using the same argument as in the proof of
  \cite[Prop.~1.2]{BH88}, we observe that as a function of $\l,$ 
$$\text{the
  function }\frac{\log^+[g_{\bfc,n}(T)]_\l}{\deg(g_{\bfc,n})}\text{  converges uniformly on 
  compact subsets of }\Aberk{\C_v}.$$ 
So, 
$$\text{the function }\frac{\log^+[g_{\bfc,n}(T)]_\l}{\deg(g_{\bfc,n})}\text{ is a continuous subharmonic function on }\Aberk{\C_v},$$
which converges to $G_{\bfc,v}$ uniformly; hence it follows from
  \cite[Prop.~8.26(c)]{Baker-Rumely} that $G_{\bfc,v}$ is continuous and
  subharmonic on $\Aberk{\C_v}.$  Furthermore, as remarked above,
  $G_{\bfc,v}$ is zero on $M_{\bfc,v}.$ 

Arguing as in the proof of Lemma~\ref{M_a is bounded for archimedean v} (see \eqref{0 equation 1} and \eqref{precise absolute value}), if $|\lambda|_v>C_4$ then for $n\ge 1$ we have
$$|g_{\bfc,n}(\l)|_v=|f_{\lambda}^{n}(\bfc(\l))|_v=|q_m\lambda^{m}|_v^{d^{n}}.$$
Hence, for $|\lambda|_v>C_4$ we have
  \begin{align*}
    G_{\bfc,v}(\l) & = \lim_{n\to \infty}\, \frac{1}{md^{n}} \log
    |g_{\bfc,n}(\l)|_v
    \\
     & = \log|\l|_v + \frac{\log |q_m|_v}{m}.
  \end{align*}
  It follows from Lemma~\ref{green function}~(3), that $G_{\bfc,v}$ is indeed
  the Green's function of $M_{\bfc,v}.$ 
\end{proof}

Now we are ready to prove Theorem~\ref{mandelbrot capacity}.
\begin{proof}[Proof of  Theorem~\ref{mandelbrot capacity}.]
As in the proof of Lemma~\ref{green's function for mandelbrot}, we
have
\[
 G_{\bfc,v}(\l)  = \log|\l|_v + \frac{\log |q_m|_v}{m} +o(1)
\]
for $|\l|_v$ sufficiently large. By Lemma~\ref{green
    function}~(3), we find that $V(M_{\bfc,v}) = \frac{\log |q_m|_v}{m}$. Hence,
  the logarithmic capacity of $M_{\bfc,v}$ is
  \[
  \g(M_{\bfc,v}) = e^{-V(M_{\bfc,v})} = \frac{1}{|q_m|^{1/m}_v}
\]
as desired. 
\end{proof}

Let $\M_{\bfc} = \prod_{v\in
\Omega} M_{\bfc,v}$ be the generalized adelic Mandelbrot set associated to
$c$. As a corollary to Theorem~\ref{mandelbrot capacity} we see that
$\M_{\bfc}$ satisfies the hypothesis of Theorem~\ref{thm:equidistribution}.
\begin{cor}
 \label{adelic mandelbrot}
For all but finitely many nonarchimedean places $v$, we have that $M_{\bfc,v}$ is the closed unit disk $\cD(0;1)$ in $\C_v$; furthermore $\g(\M_{\bfc}) = 1$.
\end{cor}

\begin{proof}
For each place $v$ where all coefficients of $c_i(\l), i=0,
\ldots,d-2$ and of $\bfc(\l)$ are
$v$-adic integral, and moreover $|q_m|_v=1$, we have that
$M_{\bfc,v}=\cD(0,1)$. Indeed, $\cD(0,1)\subset M_{\bfc,v}$ since then
$f_{\lambda}^n(\bfc(\l))$ is always a $v$-adic integer. For the converse
implication we note that each coefficient of $g_{\bfc,n}(\lambda)$ is a
$v$-adic integer, while the leading coefficient is a $v$-adic unit for
all $n\ge 1$; thus
$|g_{\bfc,n}(\lambda)|_v=|\lambda|_v^{md^{n}}\to\infty$ if
$|\lambda|_v>1$. Note that $q_m\ne 0$ and so, the second assertion in
Corollary~\ref{adelic mandelbrot} follows immediately by the product formula in $K$.
\end{proof} 

Using Proposition~\ref{prop:local height and green's function} and the decomposition of the global canonical height as a sum of local canonical heights we obtain the following result.
\begin{cor}
\label{second important remark}
Let $\l\in\Kbar$, let $S$ be the set of $\Gal(\Kbar/K)$-conjugates of $\l$, and let $h_{\M_{\bfc}}$ be defined as in \eqref{def ade hig}. Then $\deg(\bfc)\cdot h_{\M_{\bfc}}(\l)=\hhat_{f_{\l}}(\bfc(\l))$.
\end{cor}

\begin{remark}
  \label{rem:second important remark}
  Let $h(\lambda)$ denote a Weil height function corresponding to the
  divisor $\infty$ of the parameter space which is the projective line in our
  case. Then, it follows from \cite[Theorem~4.1]{Call-Silverman} that
  $$\lim_{h(\lambda)\to \infty}
  \frac{\hhat_{f_{\l}}(\bfc(\l))}{h(\lambda)} = \hhat_{\bff}(\bfc)$$
  where $\hhat_{\bff}(\bfc)$ is the canonical height associated to the
  polynomial map $\bff$ over the function field $\C(\lambda).$
  Corollary~\ref{second important remark} gives a precise relationship
  between the canonical height function on the special fiber, height
  of the parameter $\lambda$ and 
  $\hhat_{\bff}(\bfc)$ which is equal to $\deg(\bfc)$ in this case. 
\end{remark}


\section{Explicit formula for the Green function}
\label{complex analysis}


In this Section we work under the assumption  that $|\cdot |_v=|\cdot |$ is archimedean; we also denote $\C_v$ by simply $\C$ in this case.  
We show that in this setting  we have an alternative way of representing the Green's function $G_c:=G_{\bfc,v}$ for the Mandelbrot set $M_c:=M_{\bfc,v}$. We continue to work under the same hypothesis on $\bfc(\l)$; in particular we assume that \eqref{what we need for degree} holds. Furthermore, if $r=0$ (i.e., $\bff$ is a constant family of polynomials, then $m=\deg(\bfc)\ge 1$).

Since the degree in $x$ of $f_{\lambda}(x)$  is $d$, there exists a unique function $\phi_{\lambda}$ which is an analytic homeomorphism on the set $U_{R_{\l}}$ for some $R_{\l}\ge 1$ (where for any positive real number $R$, we denote by $U_R$ the open set $\{z\in\C\text{ : }|z|>R\}$) satisfying the following conditions:
\begin{enumerate}
\item $\phi_{\lambda}$ has derivative equal to $1$ at $\infty$, or more precisely, the analytic function $\psi_{\lambda}(z):=1/\phi_{\lambda}(1/z)$ has derivative equal to $1$ at $z=0$; and
\item $\phi_{\lambda}(f_{\lambda}(z))=(\phi_{\lambda}(z))^d\text{ for $|z|>R_{\lambda}$.}$
\end{enumerate}

We can make (1) above more precise by giving the power series expansion:
\begin{equation}
\label{formula for phi lambda}
\phi_{\lambda}(z)=z+ \sum_{n=1}^{\infty} \frac{A_{\lambda,n}}{z^n}.
\end{equation}
From \eqref{formula for phi lambda} we immediately conclude that $|\phi_{\lambda}(z)|= |z|+O_\l(1)$, and thus
\begin{equation}
\label{logarithm estimate for phi lambda}
\log|\phi_{\lambda}(z)|=\log|z| + O_\l(1)\quad \text{for $|z|$ large  enough}. 
\end{equation}
So, using that $\phi_{\lambda}(f_{\lambda}(z))=\phi_{\lambda}(z)^d$, we conclude that if $|z|>R_{\lambda}$, then
\begin{equation}
\label{limit estimate for phi lambda}
\lim_{n\to\infty}\frac{\log^+\left|f_{\lambda}^n(z)\right|}{d^n} =\lim_{n\to\infty} \frac{\log\left|\phi_{\lambda}(f_{\lambda}^n(z))\right|}{d^n}=\log|\phi_{\lambda}(z)|.
\end{equation}
Hence
\eqref{limit estimate for phi lambda} yields that the Green function
$G^{\lambda}$ for the (filled Julia set of the) polynomial
$f_{\lambda}$ equals  
$$G^{\lambda}(z):=\lim_{n\to\infty} \frac{\log\left|f^n_{\lambda}(z)\right|}{d^n}=\log|\phi_{\lambda}(z)|\text{, if $|z|>R_{\lambda}$.}$$
For more details on the Green function associated to any polynomial, see
\cite{Carleson-Gamelin}. Now, we know by
\cite[Ch. III.4]{Carleson-Gamelin} that the function
$\log|\phi_{\lambda}(z)|$ can be extended to a well-defined harmonic
function on the entire basin of attraction $A^{\lambda}_{\infty}$ of the
point at $\infty$ for the polynomial map $f_{\lambda}$. Thus, on
$A^{\lambda}_{\infty}$ we have that
\begin{equation}
\label{formula for the Green's function}
G^{\lambda}(z):=\log|\phi_{\lambda}(z)|
\end{equation}
is the Green function for (the
filled Julia set of) the polynomial $f_{\lambda}.$ Also by \cite[Ch. III.4]{Carleson-Gamelin} we know that  
$$R_{\lambda}:=\max_{f_{\lambda}'(x)=0} e^{G^{\lambda}(x)}\ge 1.$$

In Proposition~\ref{domain of analyticity} we will show that if $|\lambda|$ is sufficiently large, then $\bfc(\l)$ is in the domain of analyticity for $\phi_{\lambda}$. In particular, using \eqref{logarithm estimate for phi lambda} this would yield 
\begin{equation}
\label{crucial equation for the Green function}
G_{c}(\lambda)= \lim_{n\to\infty}\frac{\log^
  +\left|f_{\lambda}^{n}(\bfc(\l))\right|}{md^{n}}
=\frac{\log\left|\phi_{\lambda}(\bfc(\l))\right|}{m}
=\frac{G^\l(\bfc(\l))}{m},
\end{equation}
for $|\lambda|$ sufficiently large.

The proof of the next proposition is similar to the proof of \cite[Lemma 3.2]{Matt-Laura}.

\begin{prop}
\label{domain of analyticity}
There exists a positive constant $C_0$ such that if $|\lambda|>C_0$,
then  $\bfc(\l)$ belongs to the analyticity domain of
$\phi_{\lambda}$. 
\end{prop}

\begin{proof}
If $\bff$ is a constant family of polynomials, then the conclusion is immediate since $R_{\l}$ is constant (independent of $\l$) and thus for $|\l|$ sufficiently large, clearly $|\bfc(\l)|>R_{\l}$. So, from now on assume $\bff$ is not a constant family of polynomials, which in particular yields that $r\ge 1$ and $0<m_1<\cdots <m_r$.

First we recall that 
$$R_{\lambda}=e^{G^{\lambda}(x_0)}:=\max_{f_{\lambda}'(x)=0}e^{G^{\lambda}(x)}.$$ 
Next we show that $R_{\lambda}\to\infty$ as $|\lambda|\to\infty$, which will be used later in our proof.

\begin{lemma}
\label{radius is large}
As $|\lambda|\to\infty$, we have $R_{\lambda}\to\infty$.
\end{lemma} 

\begin{proof} We recall that
$$f_{\l}(x)=P(x)+\sum_{i=1}^r \l^{m_i} \cdot Q_i(x),$$
where $P(x)$ is a polynomial in normal form of degree $d$, and $0<m_1<\cdots < m_r$ are positive integers, while the $Q_i$'s are nonzero polynomials of degrees $e_i\le d-2$.  
We have two cases.

{\bf Case 1.} Each $Q_i(x)$ is a constant polynomial. Then the critical points of $f_{\lambda}$ are independent of $\lambda$, i.e., $x_0=O(1)$. We let $x_1\in\C$ such that $f_{\lambda}(x_1)=x_0$. Since each $Q_i$ is a nonzero constant polynomial, we immediately conclude that $|x_1|\gg |\lambda|^{m_r/d}$. On the other hand, since $U_{2R_{\lambda}}\subset \phi_{\lambda}^{-1}\left(U_{R_{\lambda}}\right)$ (by \cite[Corollary 3.3]{BH88}) we conclude that $|x_1|\le 2R_{\lambda}$, and so, $R_{\lambda}\gg |\lambda|^{m_r/d}$. Indeed, if $|x_1|>2R_{\lambda}$, then there exists $z_1\in U_{R_{\lambda}}$ such that $\phi_{\lambda}^{-1}(z_1)=x_1$. Using the fact that $\phi_{\lambda}$ is a conjugacy map at $\infty$ for $f_{\lambda}$ we would obtain that
$$x_0=f_{\lambda}(x_1)=f_{\lambda}(\phi_{\lambda}^{-1}(z_1)) =\phi_{\lambda}^{-1}(z_1^d)\in U_{R_{\lambda}},$$
which contradicts the fact that $x_0$ is not in the analyticity domain of $\phi_{\lambda}$.

{\bf Case 2.} There exists $i=1,\dots,r$ such that $Q_i(x)$ is not a constant polynomial. Then the critical points of $f_{\lambda}$ vary with $\lambda$. In particular, there exists a critical point $x_{\lambda}$ of maximum absolute value such that $|x_{\lambda}|\gg |\lambda|^{m_j/(d-e_j)}$ (for some $j=1,\dots, r$), where for each $i=1,\dots,r$ we have  $e_i=\deg(Q_i)\le d-2$. Now, $x_{\lambda}$ is not in the domain of analyticity of $\phi_{\lambda}$ and thus $|x_{\lambda}|\le R_{\lambda}$, which again shows that $R_{\lambda}\to\infty$ as $|\lambda|\to\infty$.
\end{proof}

Using that $R_{\lambda}\to\infty$, we will finish our proof. First we note that 
\begin{equation}
\label{justification}
|\phi_{\lambda}(f_{\lambda}(x_0))|=e^{G^{\lambda}(f_{\lambda}(x_0))}
= e^{d G^{\lambda}(x_0)}=R_{\lambda}^d. 
\end{equation}
Note that $\phi_{\lambda}(z)$ is analytic on $U_{R_{\lambda}}$, while
$\log|\phi_{\lambda}(z)|$ is continuous for $|z|\ge
R_{\lambda}$. Moreover, whenever it is defined,
$G^{\lambda}(f_{\lambda}(z))=d G^{\lambda}(z)$; so, using also
\eqref{formula for the Green's function}, we obtain
\eqref{justification}.  

Now, for $|\lambda|$
sufficiently large  we have that $R_{\lambda}^{d}/2>R_{\lambda}$ (since
$R_{\lambda}\to\infty$ according to Lemma~\ref{radius is large}). So,
$U_{R_{\lambda}^d}\subset
\phi_{\lambda}\left(U_{R_{\lambda}^d/2}\right)$ (again using
\cite[Corollary 3.3]{BH88}) and thus  
\begin{equation}
\label{value for x_0}
|f_{\lambda}(x_0)|\ge \frac{R_{\lambda}^d}{2}. 
\end{equation}

{\bf Case 1.} $\deg(Q_i)=0$ for each $i$. Then $x_0=O(1)$ as noticed
in Lemma~\ref{radius is large} and thus, using \eqref{value for x_0}
we obtain that $|\lambda|^{m_r}\gg R_{\lambda}^d$.  Since $\deg(\bfc)=m\ge m_r$, we obtain  
$$|\bfc(\l)|\ge |q_m|\cdot |\lambda|^{m} - |O(\l^{m-1}) |\gg
R_{\lambda}^d> R_{\lambda},$$ if $|\lambda|$ is sufficiently large. 

{\bf Case 2.} If not all of the $Q_i$'s are constant polynomials, then we still know that
$$|x_0|\ll |\lambda|^{\max_{i=1}^{r} m_i/(d-e_i)}\ll |\lambda|^{m_r/2},$$
because $e_i\le d-2$ for each $i$. 
Therefore 
\begin{equation}
\label{technicality}
R_{\lambda}^d\ll |f_{\lambda}(x_0)|\ll |\lambda|^{dm_r/2}.
\end{equation}
On the other hand, $|\bfc(\l)|\sim |\lambda|^{m}$ and $m\ge m_r$, which yields that
$$|\bfc(\l)|\gg |\lambda|^{m} \gg R_{\lambda}^2\gg R_\l,$$
by \eqref{technicality}. 
This concludes the proof of Proposition~\ref{domain of analyticity}.
\end{proof}

Therefore for large $|\lambda|$, the point $\bfc(\l)$ is in the domain of analyticity for $\phi_{\lambda}$, which allows us to conclude that equation \eqref{crucial equation for the Green function} holds. 

 We know (see \cite{Carleson-Gamelin}) that for each $\lambda\in\C$
 and for each $z\in\C$ sufficiently large in absolute value, we have: 
\begin{equation}
\label{one}
\phi_{\lambda}(z)=z\prod_{n=0}^{\infty}
\left(\frac{f_{\lambda}^{n+1}(z)}{f_{\lambda}^n(z)^d}\right)^{\frac{1}{d^{n+1}}}, 
\end{equation}
and thus
\begin{equation}
\label{two}
\phi_{\lambda}(z)=z\prod_{n=0}^{\infty} \left(1+
  \frac{Q_0\left(f_{\l}^n(z)\right)+\sum_{i=1}^{r} Q_i\left(f_{\l}^n(z)\right)\cdot
    \l^{m_i}}{f_{\lambda}^n(z)^d}\right)^{\frac{1}{d^{n+1}}}, 
\end{equation}
where $Q_0(x):=P(x)-x^d$ is a polynomial of degree at most equal to $d-2$. 
We showed in Proposition~\ref{domain of analyticity} that
$\phi_{\lambda}(\bfc(\l))$ is well-defined; furthermore the function
$\phi_{\lambda}(\bfc(\l))/\bfc(\l)$ can be expressed
near $\infty$ as 
the above infinite product. Indeed, for each $n\in\N$, the order of
magnitude of the numerator in the $n$-th fraction from the product
appearing in \eqref{two} when we substitute $z=\bfc(\l)$ is at most  
$$|\lambda|^{m+(d-2)m d^n}\le |\lambda|^{m(d-1)d^n},$$
while the order of magnitude of the denominator is $|\lambda|^{m d^{n+1}}$. This guarantees the convergence of the product from \eqref{two} corresponding to $\phi_\l(\bfc(\l))/\bfc(\l)$. We conclude that
\begin{equation}
\label{three}
\phi_\l(\bfc(\l))\text{ is an analytic function of $\l$ (for large $\l$), and moreover}
\end{equation}
\begin{equation}
\label{three bis}
\phi_\l(\bfc(\l))=q_m \lambda^{m}+O\left(\lambda^{m-1}\right).
\end{equation}


\section{Proof of Theorem~\ref{main result}: algebraic case}
\label{proof of our main result}


We work under the hypothesis of Theorem~\ref{main result}, and we continue with the notation from the previous Sections. Furthermore we prove Theorem~\ref{main result} under the extra assumptions that
\begin{equation}
\label{algebraic assumption}
\bfa,\bfb\in\Qbar[\l]\text{ and also, for each $i=0,\dots,d-2$, }c_i\in\Qbar[\l]. 
\end{equation}

Recall that 
$f_\l(x)= x^d + \sum_{i=0}^{d-2} c_i(\l) x^i$ where we require that
$c_i\in \Qbar[\l]$ for $i = 0, \ldots, d-2.$ Let $\bfa,
\bfb \in \Qbar[\l]$ satisfying the hypothesis (i)-(ii) of Theorem~\ref{main result}.  
Let $K$ be the number field generated by the
coefficients of $c_i(\l)$ for $i=0, \dots, d-2$, and of $\bfa(\l)$ and $\bfb(\l)$.  
Let $\Omega_K$ be the set of all
inequivalent absolute values on $K$.

Next, assume there exist infinitely many $\lambda$ such that both $\bfa(\l)$
and $\bfb(\l)$ are preperiodic for $f_{\lambda}$.  At the expense of replacing $\bfa(\l)$ by $f_{\l}^k(\bfa(\l))$ and replacing $\bfb(\l)$ by $f_{\l}^{\ell}(\bfb(\l))$ we may assume that the polynomials
\begin{equation}
\label{condition required}
\text{$\bfa(\l)$ and $\bfb(\l)$ have the same leading coefficient, and the same degree $m\ge m_r$.}
\end{equation}

Let $h_{\M_{\bfa}}(z)$  
($h_{\M_{\bfb}}(z)$) be the 
height of $z\in \Kbar$ relative to the adelic generalized Mandelbrot set
$\M_{\bfa} := \prod_{v\in \Omega_K} M_{\bfa,v}$ ($\M_{\bfb}$) as
defined in Section~\ref{Berkovich}. Note that if $\lambda \in \Kbar$
is a parameter such that $\bfa(\l)$ 
(and $\bfb(\l)$) is preperiodic for $f_{\lambda},$ then
$h_{\M_{\bfa}}(\lambda)=0$ by Corollary~\ref{second important remark}.
So, we may apply the
equidistribution result from \cite[Theorem 7.52]{Baker-Rumely}
(see our Theorem~\ref{thm:equidistribution}) and 
conclude that $M_{\bfa,v}=M_{\bfb,v}$ for each place $v\in \Omega_K$.
Indeed, we know that there exists an infinite sequence
$\{\l_n\}_{n\in\N}$ of distinct numbers $\l\in\Kbar$ such that both
$\bfa(\l)$ and $\bfb(\l)$ are preperiodic for $f_\l$. So, for each $n\in\N$, we may
take $S_n$ be the union of the sets of Galois conjugates for $\l_m$
for all $1\le m\le n$. Clearly $\#S_n\to\infty$ as $n\to\infty$, and
also each $S_n$ is $\Gal(\Kbar/K)$-invariant. Finally,
$h_{\M_{\bfa}}(S_n)=h_{\M_{\bfb}}(S_n)=0$ for all $n\in\N$, and thus
Theorem~\ref{thm:equidistribution} applies in this case.

Now, let $|\cdot |=|\cdot|_v$ be an archimedean absolute value in $\Omega_K$ and let $\C=\C_v$. We denote by $M_{a}:=M_{\bfa,v}$ and $M_{b}:=M_{\bfb,v}$ the corresponding Mandelbrot sets; then $M_{a}=M_{b}$ and also the corresponding Green's functions are the same, i.e. (using \eqref{crucial equation for the Green function} and \eqref{condition required})
$$\left|\phi_{\lambda}(\bfa(\l))\right|=
\left|\phi_{\lambda}(\bfb(\l))\right|  \text{ for
  all $|\lambda|$ sufficiently large.}$$ 
On the other hand, for $|z|$ large, the function
$h(z):=\phi_z(\bfa(z))/\phi_z(\bfb(z))$ is an
analytic function of constant absolute value (note that the
denominator does not vanish since $\phi_\l$ is a homeomorphism for a
neighborhood of $\infty$).    
 By the Open Mapping Theorem, we conclude that $h(z):=u$ is a constant (for some $u\in\C$ of absolute value equal to $1$); i.e.,
\begin{equation}
\label{constant quotient}
\phi_{\lambda}(\bfa(\l))=u\cdot \phi_{\lambda}(\bfb(\l)).
\end{equation}
Using \eqref{three} and \eqref{three bis} (also note that $\bfa(\l)$ and $\bfb(\l)$ have the same leading coefficient), we conclude that $u=1$.  Using that $\phi_{\lambda}$ is a homeomorphism on a neighborhood of the infinity, we conclude that $\bfa(\l)=\bfb(\l)$ for $\l$ sufficiently large in absolute value, and thus for \emph{all} $\l$, as desired (note that $\bfa$ and $\bfb$ are polynomials).

\begin{remark}
\label{some remark is important}
Our proof (similar to the proof from \cite{Matt-Laura}) only uses in an essential way the information that $M_a=M_b$, i.e., that the Mandelbrot sets over the complex numbers corresponding to $\bfa$ and $\bfb$ are equal, even though we know that $M_{\bfa,v}=M_{\bfb,v}$ for \emph{all} places $v$.
\end{remark}

\section{Proof of Theorem~\ref{main result}: the converse implication}
\label{converse}


Now we prove the converse implication in Theorem~\ref{main result} in the general case, i.e. for polynomials $c_0,\dots,c_{d-2}$, $\bfa$, and $\bfb$ with arbitrary complex coefficients. Again at the expense of replacing $\bfa(\l)$ by $f_{\l}^k(\bfa(\l))$ and replacing $\bfb(\l)$ by $f_{\l}^{\ell}(\bfb(\l))$ we may assume $\bfa(\l)=\bfb(\l)$.  The following result will finish the converse statement in Theorem~\ref{main result}.

\begin{prop}
\label{infinitely many preperiodic} 
Let $\bfc\in\C[\l]$ of degree $m\ge m_r$. Let ${\rm Prep}(\bfc)$ be the set consisting of all $\lambda\in\C$ such that $\bfc(\l)$ is preperiodic under $f_{\lambda}$, and let $M_{c}$ be the set of all $\l\in\C$ such that the orbit of $\bfc(\l)$ under the action of $f_{\l}$ is bounded with respect to the usual archimedean metric on $\C$.  
Then the closure in $\C$ of the set ${\rm Prep}(\bfc)$ contains $\partial M_c$. In particular, ${\rm Prep}(\bfc)$ is infinite. 
\end{prop}

\begin{proof}
We first claim that the equation $f_{z}(\bfc(z))=\bfc(z)$ has only finitely many solutions. Indeed, according to Lemma~\ref{lem:degree in l}, the degree in $z$ of $f_z(\bfc(z))-\bfc(z)$ is $dm$, which means that there are at most $dm$ solutions $z\in\C$ for the equation $f_z(\bfc(z))=\bfc(z)$.

Now, let $h_i:U\lra \mathbb{P}^1(\C)$ for $i=1,2,3$ be three analytic functions with values taken in the compact Riemann sphere, given by:
$$h_1(z):=\infty\text{; }h_2(z):=\bfc(z)\text{ and }h_3(z):=g_{\bfc,1}(z)=f_z(\bfc(z)).$$

Let $x_0\in\partial M_{c}$ which is \emph{not} a solution $z$ to $f_z(\bfc(z))=\bfc(z)$; we will show that $x_0$ is contained in the closure in $\C$ of ${\rm Prep}(\bfc)$. Since we already know that if $f_{z}(\bfc(z))=\bfc(z)$ then $ z\in{\rm Prep}(\bfc)$, we will be done once we prove that each open neighborhood $U$ of $x_0$ contains at least one point from ${\rm Prep}(\bfc)$. 

Furthermore, since $x_0$ is not a solution for the equation $h_2(z)=h_3(z)$, then we may assume (at the expense of replacing $U$ with a smaller neighborhood of $x_0$) that the closures of $h_2(U)$ and $h_3(U)$ are disjoint. 
Therefore  the
closures of  $h_1(U)$, $h_2(U)$ and $h_3(U)$ in $\mathbb{P}^1(\C)$ are all disjoint. 

As before, we let $\{g_{\bfc,n}\}_{n\ge 2}$ be the set of polynomials $g_{\bfc,n}(z):=f_{z}^n(\bfc(z))$. Since $x_0\in\partial M_{c}$, then the family of analytic maps $\{g_{\bfc,n}\}_{n\ge 2}$ is not normal on $U$. Therefore, by Montel's Theorem (see \cite[Theorem 3.3.6]{Beardon}), there exists $n\ge 2$ and there exists $z\in U$ such that $g_{\bfc,n}(z)=\bfc(z)$ or $g_{\bfc,n}(z)=f_{z}(\bfc(z))$ (clearly it cannot happen that $g_{\bfc,n}(z)=\infty$). Either way we obtain that $z\in\Prep(\bfc)$, as desired.

Since $\gamma(M_{c})>0$, we know that $M_{c}$ is an uncountable subset of $\C$, and thus its boundary is infinite; hence also ${\rm Prep}(\bfc)$ is infinite.
\end{proof}

\section{Proof of Theorem~\ref{main result}: general case}
\label{transcendental}


In this Section we finish the proof of Theorem~\ref{main result}. So, with the same notation as in Theorem~\ref{main result}, we replace $\bfa$ and $\bfb$ by $f_{\l}^k(\bfa(\l))$ and respectively $f_{\l}^{\ell}(\bfb(\l))$; thus $\bfa(\l)$ and $\bfb(\l)$ are polynomials of same degree and same leading coefficient. We assume there exist infinitely many $\l\in\C$ such that both $\bfa(\l)$ and $\bfb(\l)$ are preperiodic for $f_{\l}$; we will prove that $\bfa=\bfb$.

Let $K$
denote the field generated over $\Qbar$ by adjoining the coefficients of each $c_i$ (for $i=1,\dots,d-2$), and adjoining the coefficients of 
$\bfa$ and of $\bfb$. According to Corollary~\ref{always
  over the complex}, if there exists $\l\in\C$ such that $\bfa(\l)$ (or $\bfb(\l)$)
is preperiodic for $f_{\l}$, then $\l\in\overline{K}$ where
$\overline{K}$ denotes the algebraic closure of $K$ in $\C$.  
Let $\Omega_{K}$ be the set of inequivalent absolute values of $K$ corresponding to the divisors of a projective $\Qbar$-variety $\cV$ regular in codimension $1$; then the places in $\Omega_K$
satisfy a product formula.

As in Section~\ref{proof of our main result}, we let $h_{\M_\bfa}(z)$
($h_{\M_\bfb}(z)$) be the height of $z\in \Kbar$ relative to the
adelic generalized Mandelbrot set 
$\M_\bfa = \prod_{v\in \Omega_K} \bfM_{\bfa,v}$ ($\M_\bfb$) as defined in
Section~\ref{Berkovich}. Note that if $\lambda \in \overline{K}$ is a
parameter such that $\bfa(\l)$ 
is preperiodic for $f_{\lambda},$ then $h_{\M_{\bfa}}(\lambda)=0$ 
($h_{\M_{\bfb}}(\lambda)=0$ respectively) by Corollary~\ref{second
  important remark} again. So, arguing as in Section~\ref{proof of our main result}, we may apply the
equidistribution result from \cite[Theorem 7.52]{Baker-Rumely}
(Theorem~\ref{thm:equidistribution}) and 
conclude that $\bfM_{\bfa,v}=\bfM_{\bfb,v}$ for each place $v\in \Omega_K$.

As observed in our proof from Section~\ref{proof of our main result} (see Remark~\ref{some remark is important}), in order to finish the proof of Theorem~\ref{main result} it suffices to prove that $M_a=M_b$, where $M_a$ and $M_b$ are the complex Mandelbrot sets corresponding to $\bfa$ and respectively, $\bfb$.  

As before, we denote by $\Prep(\bfa)$ and $\Prep(\bfb)$ the sets of
all $\l\in\C$ such that $\bfa(\l)$ (respectively $\bfb(\l)$) is
preperiodic for $f_{\l}$. As proved in Corollary~\ref{always over the
  complex} we know that both $\Prep(\bfa)$ and $\Prep(\bfb)$ are
subsets of $\Kbar$. In order to prove that $M_a=M_b$ it suffices to
prove that $\Prep(\bfa)$ differs from $\Prep(\bfb)$ in at most
finitely many points. To ease the notation, we denote the symmetric
difference of $\Prep(\bfa)$ and $\Prep(\bfb)$ by the following:
$$
\Prepd(\bfa, \bfb) := 
\left(\Prep(\bfa)\setminus\Prep(\bfb)\right)\cup\left(\Prep(\bfb)\setminus\Prep(\bfa)\right).
$$

\begin{prop}
\label{at most finitely many bad preperiodic}
If the set $\Prepd(\bfa,\bfb)$ is finite, then $M_a=M_b$.
\end{prop}

\begin{proof}
Since $M_a$ contains all
points $\lambda\in\C$ such that $\lim_{n\to\infty}
\frac{\log^+|f_{\lambda}^n(a)|}{d^n}=0$, the Maximum Modulus Principle
yields that the complement of $M_a$ in $\C$ is connected; i.e., $M_a$
is a full subset of $\C$ (see also \cite{Matt-Laura}). So both
$M_a$ and $M_b$ are full subsets of $\C$ containing the sets
$\Prep(\bfa)$ and  $\Prep(\bfb)$ respectively whose closures contain
the boundary of $M_{a}$ and respectively, of $M_{b}$ (according to
Proposition~\ref{infinitely many preperiodic}). As $\Prep(\bfa)$ and
$\Prep(\bfb)$ differ by at most finitely many elements,  we conclude that $M_{a}=M_{b}$.
\end{proof}

In order to prove that $\Prep(\bfa)$ and $\Prep(\bfb)$ differ by at
most finitely many elements, we observe first that if
$\l\in\Prep(\bfa)$, then $\hhat_{f_{\l}}(\bfa(\l))=0$ and thus
$\l^{\sigma}\in M_{\bfa,v}$ for all $v$ and all
$\sigma\in\Gal(\Kbar/K)$ (see \eqref{isotrivial thing}; note that
$\bfa(\l)^{\sigma}=\bfa\left(\l^{\sigma}\right)$ since $\bfa\in
K[x]$).  Similarly, if $\l\in\Prep(\bfb)$ then $\l^{\sigma}\in
M_{\bfb,v}$ for each place $v\in\Omega_K$ and each Galois morphism
$\sigma$. We would like to use the reverse implication, i.e.,
characterize  the elements $\Prep(\bfa)$ as the set of all
$\l\in\Kbar$ such that $\l^{\sigma}\in M_{\bfa,v}$ for each place $v$
and for each Galois morphism  $\sigma$. This is true if 
$f_\lambda$ is not isotrivial over $\Qbar$ by Benedetto's result
\cite{Rob}. In this case, $\Prep(\bfa)$ ($\Prep(\bfb)$) is exactly the
set of $\lambda\in \Kbar$ such that  $h_{\M_\bfa}(\lambda) = 0$
($h_{\M_\bfa}(\lambda) = 0$ respectively). 
However, notice that  if $f_{\l}\in\Qbar[x]$ then   
$$\l^{\sigma}\in M_{\bfa,v}\text{ for all $v\in\Omega_K$ and
  $\sigma\in\Gal(\Kbar/K)$ if and only if }\bfa(\l)\in\Qbar.$$  We see
that in this case $\Prep(\bfa)$ is strictly smaller than the set of  $\lambda\in
\Kbar$ such that  $h_{\M_\bfa}(\lambda) = 0.$ 
So, we will prove that $\Prep(\bfa)$ and $\Prep(\bfb)$ differ by at
most finitely many elements by splitting our analysis into two cases
depending on whether there exist infinitely many $\l\in\C$ such that
$f_{\l}$ is conjugate to a polynomial with coefficients in $\Qbar$.

The following easy result is key for our extension to $\C$.

\begin{lemma}
\label{not over Qbar}
For any $\lambda\in\C$, the polynomial $f_{\lambda}(x)$ is conjugate
to a polynomial with coefficients in $\Qbar$ if and only if $c_i(\l)\in \Qbar$ for each $i=1,\dots,d-2$.
\end{lemma}

\begin{proof}
  One direction is obvious.  Now, assume $f_{\l}$ is conjugate to a
  polynomial with coefficients in $\Qbar$. Let $\delta(x) := ax + b$
  be a linear polynomial such that $\delta^{-1}\circ f_{\l}\circ
  \delta \in \Qbar[x]$. Since $f_{\l}$ is in normal form, we note that
  $a, b\in \Qbar$ for otherwise the leading coefficient or the
  next-to-leading coefficient is not algebraic. Now, it is clear that
  each $c_i(\l)\in\Qbar$ as desired.
\end{proof}

Let $S$ be the set of all $\l\in\C$ such that $f_{\l}$ is conjugate to a polynomial in $\Qbar[x]$. Using Lemma~\ref{not over Qbar}, $S\subset\Kbar$ since each polynomial $c_i$ has coefficients in $K$, and $\Qbar\subset K$. Also, $S$ is $\Gal(\Kbar/K)$-invariant since each coefficient of each $c_i$ is in $K$.

\begin{prop}
\label{finitely many lambdas bad}
We have $$\Prepd(\bfa, \bfb)\subset S.$$
\end{prop}

\begin{proof}
  Let $\l\in\Kbar\setminus S$. Since $f_{\l}$ is not conjugate to a
  polynomial in $\Qbar[x]$, using Benedetto's result (see also
  \eqref{isotrivial thing}) we obtain that $\bfa(\l)$ is preperiodic
  for $f_{\lambda}$ if and only if   for each $v\in\Omega_{K}$ and
  $\sigma\in\Gal(\overline{K}/K)$, the local canonical height of
  $\bfa(\l)^{\sigma}=\bfa\left(\l^{\sigma}\right)$ computed with
  respect to $f_{\lambda}^{\sigma}$ equals $0$. Since each coefficient
  of $c_i(\l)$ is defined over $K$, we get that
  $f_{\lambda}^{\sigma}=f_{\lambda^{\sigma}}$. In conclusion, for each
  $\lambda \in \overline{K}\setminus S$, $\bfa(\lambda)$ (or
  $\bfb(\lambda)$) is preperiodic for $f_{\lambda}$ if and only if for
  each $v\in\Omega_K$ and for each $\sigma\in\Gal(\Kbar/K)$, we have
  $\lambda^{\sigma}\in M_{\bfa,v}$ (respectively $\lambda^{\sigma}\in
  M_{\bfb,v}$).  Because $M_{\bfa,v}=M_{\bfb,v}$ for all
  $v\in\Omega_K$, we conclude that for $\lambda\in\Kbar\setminus S$ then
  $\l\in \Prep(a)$ if and only if  $\l\in \Prep(b)$. Hence, 
  $\Prepd(\bfa,\bfb)\subset S$ as desired. 
\end{proof}

Next, we have the following observation. 

\begin{lemma}
  \label{lem:transcendental for bfa}
  If $\l\in S$ such that $\bfa(\l)\not\in \Qbar$ then $\bfa(\l)$ is
  not preperiodic for $f_\l$.
\end{lemma}

\begin{proof}
The assertion is immediate since for $\l\in S$ we have $f_\l\in \Qbar[x]$ by
the definition of $S$, it follows that the set of preperiodic points
of $f_\l$ is contained in $\Qbar$. By assumption $\bfa(\l)\not\in \Qbar$, 
therefore $\bfa(\l)$ is not preperiodic for $f_\l$.   
\end{proof}

\begin{prop}
\label{infinitely many lambdas bad}
$\Prepd(\bfa,\bfb)$ is  a finite set.
\end{prop}

\begin{proof}
  If $S$ is a finite set, then the assertion follows from
  Proposition~\ref{finitely many lambdas bad}. So, in the remaining
  part of the proof, we assume that $S$ is an infinite set. 
By Lemma~\ref{not over Qbar} we know that there exist infinitely many $\l\in\Kbar$ such that $c_i(\l)\in\Qbar$ for each $i=0,\dots, d-2$. 
The following Lemma will be key for our proof.
\begin{lemma}
\label{algebraic geometry}
Let $L_1\subset L_2$ be algebraically closed fields of characteristic $0$, and let $f_1,\dots,f_n\in L_2[x]$. If there exist infinitely many $z\in L_2$ such that $f_i(z)\in L_1$ for each $i=1,\dots,n$, then there exists $h\in L_2[x]$, and there exist $g_1,\dots,g_n\in L_1[x]$ such that $f_i=g_i\circ h$ for each $i=1,\dots, n$. 
\end{lemma}

\begin{proof}
Let $C\subset \bA^n$ be the Zariski closure of the set
\begin{equation}
\label{parametrization 1}
\{(f_1(z),\cdots, f_n(z))\text{ : }z\in L_2\}.
\end{equation}
Then $C$ is a rational curve which (by our hypothesis) contains infinitely many points over $L_1$. Therefore $C$ is defined over $L_1$, and thus it has a rational parametrization over $L_1$. Let 
$$(g_1,\dots,g_n):\bA^1\lra C$$
be a birational isomorphism defined over $L_1$; we denote by $\psi:C\lra \bA^1$ its inverse (for more details see \cite[Chapter 1]{Shafarevich}).    
Since the closure of $C$ in $\bP^n$ (by considering the usual embedding of $\bA^n\subset \bP^n$) has only one point at infinity (due to the  parametrization \eqref{parametrization 1} of $C$), we conclude that (perhaps after  a change of coordinates) we may assume each $g_i$ is also a polynomial; more precisely, $g_i\in L_1[x]$. We let $h:\bA^1\lra \bA^1$ be the rational map (defined over $L_2$) given by the composition
$$h:=\psi\circ (f_1,\dots,f_n).$$
Therefore, for each $i=1,\dots, n$, we have $f_i=g_i\circ h$, and since both $f_i$ and $g_i$ are polynomials, we conclude that also $h$ is a polynomial, as desired.
\end{proof}
As an immediate consequence of Lemma~\ref{algebraic geometry}, we have
the following result.
\begin{cor}
\label{cor: algebraic geometry}
Under the hypothesis from Lemma~\ref{algebraic geometry}, for each $i,j\in\{1,\dots,n\}$ and for each $z\in L_2$ we have that $f_i(z)\in L_1$ if and only if $f_j(z)\in L_1$ (if and only if $h(z)\in L_1$).
\end{cor}

There are two possibilities: either there exist infinitely many $\l\in S$ such that $\bfa(\l)\in \Qbar$, or not.

\begin{lemma}
\label{infinitely many bad lambdas for bfa}
If there exist infinitely many $\l\in S$ such that $\bfa(\l)\in \Qbar$, then $\bfa=\bfb$. In particular,  $\Prep(\bfa)=\Prep(\bfb)$.
\end{lemma}

\begin{proof}
Using Corollary~\ref{cor: algebraic geometry} we obtain that actually for \emph{all} $\l\in S$ we have that $\bfa(\l)\in\Qbar$. So, in this case \emph{each} $\l^{\sigma}$ belongs to each $M_{\bfa,v}$ for each place $v$ of the function field $K/\Qbar$ and for each $\sigma\in\Gal(\Kbar/K)$ (note that for such $\l\in S$ we have that both $f_{\l}\in\Qbar[x]$ and $\bfa(\l)\in\Qbar$, and also note that $S$ is $\Gal(\Kbar/K)$-invariant). Since $M_{\bfa,v}=M_{\bfb,v}$ for each place $v$, we conclude that $\l^{\sigma}\in M_{\bfb,v}$ for each $\l\in S$, for each $v\in\Omega_K$ and for each $\sigma\in\Gal(\Kbar/K)$. Since $f_{\l}\in\Qbar[x]$, we conclude that $\bfb(\l)\in\Qbar$ as well. Hence both $\bfa(\l)\in\Qbar$ and $\bfb(\l)\in\Qbar$ for $\l\in S$. 

Therefore, applying Lemma~\ref{algebraic geometry} to the polynomials $c_0,\dots,c_{d-2}$, $\bfa$ and $\bfb$, we conclude that there exist polynomials  polynomials $c'_0,\dots,c'_{d-2}, \bfa ',\bfb ' \in \Qbar[x]$ and $h\in \Kbar[x]$ such that 
\begin{equation}
\label{rel-one}
c_i=c'_i\circ h\text{ for each $i=0,\dots,d-2$, and}
\end{equation}
\begin{equation}
\label{rel-two}
\bfa = \bfa '\circ h\text{ and }\bfb = \bfb '\circ h.
\end{equation}
We let $\delta:=h(\l)$, and define the family of polynomials
$$f'_{\delta}(x):=x^d+\sum_{i=0}^{d-2} c'_i(\delta) x^i.$$
So,  we reduced the problem to the case studied in Section~\ref{proof of our main result} for the family of polynomials $f_{\delta}'\in\Qbar[x]$ and to the starting points $\bfa ', \bfb '\in\Qbar[\delta]$. Note that using hypothesis (i)-(ii) from Theorem~\ref{main result}, and also relations \eqref{rel-one} and \eqref{rel-two}, $\bfa '(\delta)$ and $\bfb '(\delta)$  have the same leading coefficient and the same degree which is larger than the degrees of the $c_i'$'s. 
So, since we know there exist infinitely many $\delta\in\C$ such that $\bfa '(\delta)$ and $\bfb '(\delta)$ are both preperiodic for $f'_{\delta}$ we conclude that $\bfa ' = \bfb '$, as proved in Section~\ref{proof of our main result}. Hence $\bfa=\bfb$ and thus $\Prep(\bfa)=\Prep(\bfb)$.
\end{proof}

\begin{lemma}
\label{finitely many bad lambdas for bfa}
If there exist finitely many $\l\in S$ such that $\bfa(\l)\in \Qbar$, then   $\Prepd(\bfa,\bfb)$ is finite.
\end{lemma}

\begin{proof}
We observe that there exist also at most finitely many $\l\in S$ such
that $\bfb(\l)\in\Qbar$. Otherwise, arguing as in the proof of
Lemma~\ref{infinitely many bad lambdas for bfa}, we would obtain that
for all (the infinitely many) $\l\in S$ both $\bfa(\l)$ and $\bfb(\l)$
are in $\Qbar$, thus contradicting the hypothesis of our Lemma. So,
let $T$ be the finite subset of $S$ containing all $\l$ such that
either $\bfa(\l)\in\Qbar$ or $\bfb(\l)\in\Qbar$.

Let $\l\in\Kbar\setminus T$ be such that $\l\in \Prep(\bfa).$ Then,
either $\l\not\in S$ or $\l\in S.$ 
If $\l\not\in S$ then by Lemma~\ref{finitely many lambdas bad} we also
  have $\l\in \Prep(\bfb)$. On the other hand, if $\l\in S\setminus T$
  then by Lemma~\ref{lem:transcendental for bfa} we know that
  $\l\not\in\Prep(\bfa)$, contradicting the choice of $\l.$ In
  conclusion, we find that for $\l\in\Kbar\setminus T$ we have $\l\in
  \Prep(\bfa)$ if and only if $\l\in \Prep(\bfb)$. Therefore,
  $\Prepd(\bfa,\bfb) \subset T$. Since $T$ is finite, we conclude our proof for Lemma~\ref{finitely many bad lambdas for bfa}.




\end{proof}
Lemmas~\ref{infinitely many bad lambdas for bfa} and \ref{finitely many bad lambdas for bfa} finish the proof of Proposition~\ref{infinitely many lambdas bad}.
\end{proof}

Therefore  Proposition~\ref{infinitely many lambdas bad} yields that $\Prep(\bfa)$ and $\Prep(\bfb)$ differ by at most finitely many elements. 
Then it follows from Proposition~\ref{at most finitely many bad preperiodic}  that the corresponding complex Mandelbrot sets $M_a$ and $M_b$ are equal, and so we conclude our proof of Theorem~\ref{main result} as in our proof from Section~\ref{proof of our
main result} using the equality between the two Green's functions for $M_a$ and respectively $M_b$.


\section{Connections to the dynamical Manin-Mumford conjecture}
\label{smc}


We first prove Corollary~\ref{small points corollary} and then we present further connections between our Question~\ref{main conjecture} and the Dynamical Manin-Mumford Conjecture  formulated by Ghioca, Tucker, and Zhang in \cite{IMRN}.

\begin{proof}[Proof of Corollary~\ref{small points corollary}.]
At the expense of replacing $f$ by a conjugate $\delta^{-1}\circ f\circ \delta$, and replacing $\bfa$ (resp. $\bfb$) by $\delta^{-1}\circ\bfa$ (resp. $\delta^{-1}\circ\bfb$) we may assume $f$ is in normal form. By the hypothesis of Corollary~\ref{small points corollary} we know that there are infinitely many $\l_n\in\Qbar$ such that 
$$\lim_{n\to\infty}\hhat_{f}(\bfa(\l_n))+\hhat_{f}(\bfb(\l_n))=0.$$

We let $\bff:=f_{\l}:=f$ be the constant family of polynomials $f$ indexed by $\l\in\Qbar$.  
As before, we let $K$ be the field generated by coefficients of $f$,
$\bfa$ and $\bfb$ and let $h_{\M_{\bfa}}(z)$  
($h_{\M_{\bfb}}(z)$) be the 
height of $z\in \Kbar$ relative to the adelic generalized Mandelbrot set
$\M_{\bfa} := \prod_{v\in \Omega_K} M_{\bfa,v}$ ($\M_{\bfb}$) as defined in Section~\ref{Berkovich}. So, we may apply the
equidistribution result from \cite[Theorem 7.52]{Baker-Rumely}
(see our Theorem~\ref{thm:equidistribution}) and 
conclude that $M_{\bfa,v}=M_{\bfb,v}$ for each place $v\in \Omega_K$.
Indeed,  for each $n\in\N$, we may
take $S_n$ be the set of Galois conjugates of $\l_n$. Clearly $\#S_n\to\infty$ as $n\to\infty$ (since the points $\l_n$ are distinct and their heights are bounded because the heights of $\bfa(\l_n)$ and $\bfb(\l_n)$ are bounded). Finally,
$\lim_{n\to\infty}h_{\M_{\bfa}}(S_n)=\lim_{n\to\infty}h_{\M_{\bfb}}(S_n)=0$ (by Corollary~\ref{second important remark}), and thus
Theorem~\ref{thm:equidistribution} applies in this case. 

Using that $M_{\bfa,v}=M_{\bfb,v}$ for an archimedean place $v$, the same argument as in the proof of Theorem~\ref{main result} yields that $\bfa=\bfb$, as desired.
\end{proof}

Next we discuss the connection between our Question~\ref{main conjecture} and the Dynamical Manin-Mumford Conjecture \cite[Conjecture 1.4]{IMRN}.
\begin{conjecture}[Ghioca, Tucker, Zhang]
\label{IMRN conjecture}
Let $X$ be a projective variety, let $\varphi: X
  \lra X$ be
  an endomorphism defined over $\C$ with a polarization, and let $Y$ be a subvariety of $X$ which has no component included into the singular part of $X$. Then $Y$ is preperiodic under $\varphi$
   if and only if there exists a Zariski dense subset of smooth points $x\in Y\cap \Prep_{\varphi}(X)$ such that the tangent subspace of $Y$ at $x$ is preperiodic under the induced action of $\varphi$ on the Grassmanian ${\rm Gr}_{\dim(Y)}\left(T_{X,x}\right)$. (Here we denote by $T_{X,x}$ the tangent space of $X$ at the point $x$.)
\end{conjecture}

In \cite{IMRN}, Ghioca, Tucker, and Zhang prove that
Conjecture~\ref{IMRN conjecture} holds whenever $\Phi$ is a
polarizable algebraic group endomorphism of the abelian variety $X$,
and also when $X=\bP^1\times\bP^1$, $Y$ is a line, and
$\Phi(x,y)=(f(x),g(y))$ for any rational maps $f$ and $g$.  We claim
that a positive answer to Question~\ref{main conjecture} yields the
following special case of Conjecture~\ref{IMRN conjecture} which is
not covered by the results from \cite{IMRN}. Note that we do not need
the condition on preperiodicity of tangent spaces in the Grassmanian,
only an infinite family of preperiodic points; hence what one would obtain
here is really a special case of Zhang' original dynamical
Manin-Mumford conjecture (which did not require the extra hypothesis on
tangent spaces).  

\begin{prop}
\label{implication between conjectures}
If Question~\ref{main conjecture} holds, then for any endomorphism
$\Phi$ of $\bP^1\times\bP^1$ given by $\Phi(x,y):=(f(x),f(y))$ for some
rational map $f\in\C[x]$ of degree at least $2$, a curve $Y
\subset \bP^1 \times \bP^1$ will contain infinitely many preperiodic
points if and only if $Y$ is preperiodic under $\Phi$.  In particular,
Question~\ref{main conjecture} implies Conjecture~\ref{IMRN conjecture} for such
$Y$ and $\Phi$.
\end{prop}

\begin{proof}
  Let $Y\subset \bP^1\times \bP^1$ be a curve containing infinitely
  many points $(x,y)$ such that both $x$ and $y$ are preperiodic for
  $f$. Furthermore, we may assume $Y$ projects dominantly on each
  coordinate of $\bP^1\times \bP^1$ since otherwise it is immediate to
  conclude that $Y$ contains infinitely many preperiodic points for $\Phi$ if and only if $Y=\{c\}\times\bP^1$, or $Y=\bP^1\times \{c\}$, where
  $c$ is a preperiodic point for $f$.

  We let $\bff=f_{\l}:=f$ be the constant family of rational functions
  (equal to $f$) indexed by all points $\l\in Y$, and let $K$ be the function field of $Y$. Let $(\bfa,\bfb)\in
  \bP^1(K)\times \bP^1(K)$ be a generic point for $Y$. By our
  assumption, there exist infinitely many $\l\in Y$ such that both
  $\bfa(\l)$ and $\bfb(\l)$ are preperiodic for $f_{\l}=f$. Since $Y$
  projects dominantly on each coordinate of $\bP^1\times \bP^1$, we
  get that neither $\bfa$ nor $\bfb$ is preperiodic under the action
  of $f$ (otherwise, $\bfa$ or $\bfb$ would be constant). So, assuming Question~\ref{main conjecture} holds, we obtain that the curve $Y(\bC)=\{(\bfa(\l),\bfb(\l))\text{ : }\lambda \in Y\}\subset\bP^1_K\times_K \bP^1_K$ lies on a preperiodic proper subvariety $Z$ of $\bP^1\times\bP^1$ defined over a finite extension of $K$. More precisely, we get that  $Z=Y\otimes_{\C}K $ and so, $Y$ must be itself preperiodic under the action of $(f,f)$ on $\bP^1\times\bP^1$.

Conversely, suppose that $Y$ is preperiodic under $\Phi$.  Then some
iterate of $Y$ contains a dense set of periodic points, by work of
Fakhruddin \cite{Fa}, so $Y$ contains an infinite set of preperiodic
points.  
\end{proof}

\begin{remark}
\label{important remarks at the end}
(a) In the proof of Proposition~\ref{implication between conjectures}
we did not use the full strength of the hypothesis from
Conjecture~\ref{IMRN conjecture}. Instead we used the weaker
hypothesis of \cite[Conjecture 2.5]{Zhang} or \cite[Conjecture 1.2.1,
Conjecture 4.1.7]{Zhang-lec} (which was the original formulation of
the Dynamical Manin-Mumford Conjecture). This is not surprising since
for curves contained in $\bP^1\times \bP^1$ the \emph{only}
counterexamples to the original formulation of the Dynamical
Manin-Mumford Conjecture are expected to occur when $\Phi:=(f,g)$ for
two \emph{distinct} Latt\`{e}s maps. 
\\
(b) Finally, we note that a positive answer to Conjecture~\ref{IMRN
  conjecture} does not yield a positive answer to Question~\ref{main
  conjecture}. Instead, Question~\ref{main conjecture} goes in a different direction which is likely to shed more light on the Dynamical Manin-Mumford Conjecture, especially in the case $Y$ is a curve  in Conjecture~\ref{IMRN conjecture}. 
\end{remark}

\section{Further extensions}
\label{future}

Many of the techniques here rely crucially on equidistribution results of
Baker-Rumely \cite{Baker-Rumely}.  Yuan \cite{Yuan} has generalized these results
to higher dimensions, using slightly different terminology, which was
developed by Chambert-Loir \cite{CL}.  More recently, Yuan and Zhang \cite{YZ} have
combined these equidistribution results with a $p$-adic version of the
Calabi theorem to obtain a result that says that any height function
corresponding to a semipositive metrized line bundle of fixed degree is uniquely
determined as a function by any dense set of points at which the
function is sufficiently small (in many applications, sufficiently
small means tending to 0). 

The relevant height functions for our purposes will be of the form
$h_a(f) := \hhat_f(a)$ where $f$ is a polarizable morphism $f: X \lra X$ on
a projective variety, $a$ is a point on $X$, and $\hhat_f$ is the canonical
height associated to $f$.  Thus, one uses canonical heights on some
variety $X$ to get heights on some moduli of morphisms of $X$.  If one
can show that each height function $h_a$ comes from a semipositive
metrized line bundle, then the results of Yuan-Zhang apply.  When this
is the case, one obtains that if $a$ and $b$ are points on $X$ such
that $h_a(f) = h_b(f) = 0$ for some dense set of morphisms $f$ in the
moduli space, then $h_a$ and 
$h_b$ must be proportional; this implies in particular that for \emph{each}
morphism $f$ in our moduli space, $a$ is preperiodic for $f$ if and only if
$b$ is preperiodic for $f$.

The heights $h_a$ described above are obtained via an iterative limit
process.  The main obstacle to applying the results of Yuan-Zhang
seems to be showing that these limits converge uniformly (over the
moduli of morphisms $f$); in most cases this is very difficult to prove.

Even in the case of dimension 1, the Chambert-Loir/Yuan-Zhang approach
to equidistribution seem to give greater flexibility than the methods
used here. We will use this approach to treat more general families of
rational functions in an upcoming paper; it is not clear how one might
treat these families using the methods of this paper.  However, a significant 
advantage of the method used in the present paper (relative to the proposed
new method) is that we are able here to give explicit necessary and
sufficient conditions for the starting points $a$ and $b$ guaranteeing
the existence of infinitely many functions $f$ in the moduli space of
endomorphisms of $\mathbb{A}^1$ such that both $a$ and $b$ are
preperiodic points for $f$. Most of the time, depending on the moduli
space of morphisms $\mathcal{F}$, the method outlined above may \emph{only}
yield that if there is a dense set of $f\in\mathcal{F}$ such that both
$a$ and $b$ are preperiodic for $f$, then for \emph{all}
$f\in\mathcal{F}$, $a$ is preperiodic for $f$ if and only if $b$ is
preperiodic for $f$. Finding explicit conditions for $a$ and $b$ such
as in our Theorem~\ref{main result} turns out to be very difficult for
general moduli space $\mathcal{F}$. Also, as previously mentioned, it is quite difficult to check that a given family of morphisms verifies the technical hypothesis needed in order to apply the equidistribution results of Yuan-Zhang; usually one needs a case-by-case argument for each family. However, in the present paper,  using the equidistribution result  of Baker-Rumely we give a unified treatment for almost all families of polynomials in normal form.

On the other hand, a great advantage of the new method is that it
would treat some higher-dimensional examples, though some technical
issues have yet to be resolved. The equidistribution results of Baker
and Rumely \cite{Baker-Rumely} are strictly confined to the case when
the moduli space is $\mathbb{P}^1$. If the moduli space is a curve of positive genus, then one has to apply the equidistribution results of Chambert-Loir \cite{CL}, while if the moduli space has dimension larger than $1$, then one has to apply the results of Yuan and Zhang \cite{Yuan, YZ}.

We note that in general one also expects this outlined new approach to yield
``Bogomolov type'' results for points of small height (not just
preperiodic points), provided everything is defined over a number
field.  The approach of Masser and Zannier \cite{M-Z-1, M-Z-2}
typically yields ``Manin-Mumford-type'' results (for preperiodic
points, not points of small height).





\end{document}